\documentclass[a4paper,10pt, reqno]{amsart}
\usepackage{amsmath,amsthm,amssymb,amscd,graphicx}
\usepackage[latin1]{inputenc}
\usepackage[T1]{fontenc}
\usepackage[english]{babel}
\usepackage[all]{xy}
\usepackage[pdftex]{hyperref}

\makeatletter
 \newtheorem{thm}{Theorem}[section]
 \newtheorem{lem}[thm]{Lemma}
 \newtheorem*{thm*}{Theorem}
 \newtheorem{cor}[thm]{Corollary}
 \newtheorem{prop}[thm]{Proposition}

 \theoremstyle{definition}
  \newtheorem{defin}{Definition}[section]
 \newtheorem*{exp*}{Example}
 \newtheorem{ex}{Example}[section]
 \newtheorem{rem}{Remark}[section]
  \newtheorem*{rem*}{Remark}

\def\ggr{\textrm{\emph{gr}}}
\def\C{{\mathbb C}}

\def\N{{\mathbb N}}

\def\R{{\mathbb R}}
\def\Z{{\mathbb Z}}
\def\K{{\mathbb K}}

\def\E{{\mathbb E}}
\def\V{{\mathbb V}}

\def\F{{\mathbb F}}

\def\g{{\mathfrak g}}
\def\p{{\mathfrak p}}

\def\h{{\mathfrak h}}

\def\b{{\mathfrak b}}

\def\n{{\mathfrak n}}
\def\z{{\mathfrak z}}
\def\s{{\mathfrak s}}

\def\G{{\mathcal G}}
\def\U{{\mathcal U}}
\def\T{{\mathcal T}}
\def\I{{\mathcal I}}

\def\J{{\mathcal J}}

\def\<{\langle}
\def\gr{{\textrm{gr}}}
\def\im{{\textrm{im}}}

\title{Prolongation on regular infinitesimal flag manifolds}
\author{Katharina Neusser}
\address{Katharina Neusser, Mathematical Sciences Institute, Australian National University, ACT 0200, Australia}
\email{katharina.neusser@anu.edu.au}


\begin{document}
\maketitle

\begin{abstract}
Many interesting geometric structures can be described as regular infinitesimal flag structures, which occur as the underlying structures of parabolic geometries. Among these structures we have for instance conformal structures, contact structures, certain types of generic distributions and partially integrable almost CR-structures of hypersurface type. The aim of this article is to develop for a large class of (semi-) linear overdetermined systems of partial differential equations on regular infinitesimal flag manifolds $M$ a conceptual method to rewrite these systems as systems of the form $\tilde\nabla(\Sigma)+C(\Sigma)=0$, where $\tilde\nabla$ is a linear connection on some vector bundle $V$ over $M$ and $C: V\rightarrow T^*M\otimes V$ is a (vector) bundle map. In particular, if the overdetermined system is linear, $\tilde\nabla+C$ will be a linear connection on $V$ and hence the dimension of its solution space is bounded by the rank of $V$. We will see that the rank of $V$ can be easily computed using representation theory.
\end{abstract}

\section{Introduction}
Given some overdetermined system of linear partial differential equations on a manifold $M$, one can ask the question whether this system can be rewritten as a first order system in closed form, meaning that all first order partial derivatives of the dependent variables are expressed in the dependent variables themselves. To prolong an overdetermined system in this way actually demands to introduce new variables for certain unknown higher partial derivatives until all first order partial derivatives of all the variables can be obtained as differential consequences of the original system of equations. This can be rephrased as the need to construct a vector bundle $V$ over $M$ and a linear connection on $V$ such that its parallel sections correspond bijectively to solutions of the original system of equations, see \cite{BCEG}. 
Having rewritten a system of linear differential equations in this way gives considerable information about the system. Namely, it implies that the dimension of the solution space is bounded by the rank of $V$ and by looking at the curvature of the linear connection and its covariant derivatives one may derive obstructions to the existence of solutions. In \cite{Spencer} Spencer studies a class of systems of linear differential equations, namely systems of so called finite type.
For a system of differential equations of finite type it can be shown that a solution is already determined by a finite jet in a single point. Hence this is a class of systems of differential equations, for which one can expect such a rewriting procedure to exist. However, even for simple differential equations, where one maybe is also able to check easily that they are of finite type, to rewrite them as a first order closed systems can become quite involved, see \cite{BCEG} and \cite{E}. 
Also, there is in general no conceptual method telling one how to proceed. 
\\In \cite{BCEG} Branson, \v Cap, Eastwood and Gover developed such a method for a huge class of overdetermined systems of finite type on manifolds endowed with an almost hermitian symmetric structure. Our aim is to generalise this prolongation procedure to a broader class of geometric structures.
\\Suppose that $M$ is a manifold endowed with the geometric structure of a filtered manifold, meaning that its tangent bundle $TM$ is filtered by vector subbundles $TM=T^{-k}M\supset ...\supset T^{-1}M$ such that the tangential filtration is compatible with the Lie bracket of vector fields. Then the Lie bracket of vector fields induces a tensorial bracket on the associated graded vector bundle $\gr(TM)=\bigoplus_{i=1}^kT^{-i}M/T^{-i+1}M$, making each fiber $\gr(T_xM)$ over a point $x\in M$ into a nilpotent graded Lie algebra, which should be seen as the first order approximation to the filtered manifold at $x\in M$, replacing the role of the tangent space for ordinary manifolds. 
\\One of the best studied examples of a filtered manifold is the case, where $M$ is a contact manifold $T^{-1}M\subset TM$, i.e. $\gr(T_xM)$ is a Heisenberg Lie algebra. 
Studying analytic properties of differential operators on a contact manifold $M$, it was already observed in the 70's and 80's of the last century that the notion of order for differential operators on $M$ should be better changed and adapted to the contact structure by considering a derivative in direction of a vector field transversal to the contact subbundle $T^{-1}M$ as a differential operator of order two rather than one. Adjusting the notion of order in this way leads then to a notion of weighted symbol for differential operators on $M$, which naturally fits together with the contact structure and can be viewed, in contrast to the usual principal symbol, as the 'principal part' of a differential operator on $M$, see \cite{BG} and \cite{Tay}. 
\\Independently of these developments in contact geometry, Morimoto  started in the 90's to study differential equations on general filtered manifolds and developed a formal theory, see \cite{Morimoto1}, \cite{Morimoto2} and  \cite{Morimoto3}.  By adjusting the notion of order of differentiation to the filtration of a filtered manifold, he introduced a concept of weighted jet bundles and suggested it as a convenient framework to investigate differential operators between sections of vector bundles over a filtered manifold.
\\Studying the problem of prolongation of differential equations on a filtered manifold $M$, it turns out that there exists a lot of examples of linear differential equations, for which a solution is already determined by finitely many partial derivatives in a single point, but which are not of finite type in the classical sense of Spencer. This indicates that prolongation of differential equations on filtered manifolds should be studied within the framework of weighted jet bundles and the notion of finite type should be adjusted to the weighted setting, see also \cite{N} and \cite{thesis}. 
\\For a semisimple Lie group $G$ and a parabolic subgroup $P$ a regular infinitesimal flag manifold of type $(G,P)$ is a certain type of filtered manifold $M$ together with some reduction of the structure group of the frame bundle of $\gr(TM)$ to the Levi subgroup $G_0$ of $P$. These geometric structures occur as underlying structures of parabolic geometries, which have been intensively studied in the last decades, see \cite{CSbook}. In this article we shall study a broad class of semi-linear overdetermined systems on regular infinitesimal flag manifolds by working within the setting of weighted jet bundles and we will establish an explicit prolongation procedure for them. If the tangential filtration of the regular infinitesimal flag structure is trivial, it is an almost hermitian symmetric structure and our method will recover the one in \cite{BCEG}.
\\Having introduced the concept of weighted jet bundles and the necessary background on regular infinitesimal flag structures, we will first consider semi-linear systems on regular infinitesimal flag manifolds $M$, where $G_0$ has one dimensional center.  We will show how for a huge class of such systems, which will turn out to be of weighted finite type, one can construct a linear connection $\tilde\nabla$ on some vector bundle $V$ and a bundle map $C: V\rightarrow T^*M\otimes V $ such that solutions of $\tilde\nabla(\Sigma)+C(\Sigma)=0$ correspond to solutions of the studied system. If the system is linear, $C$ will be a vector bundle map and we will obtain a correspondence of solutions of the system in question and parallel sections of the linear connection $\tilde\nabla+C$.
Further, we will apply our results to the case of contact manifolds, which will lead to an alternative prolongation method as the one developed by Eastwood and Gover in \cite{EG}.
Finally, we will discuss in the last part of this article the case of general regular infinitesimal flag manifolds and the changes with respect to the results in the case of $G_0$ having one dimensional center.
\\\\\textbf{Acknowledgments}
I would like to thank Andreas \v Cap for several helpful discussions and suggestions. Also I am grateful to Michael Eastwood and Tohru Morimoto for valuable conversations.  
This work was supported by the Initiativkolleg IK-1008 of the University of Vienna and by the project P 19500-N13 of the "Fonds zur F\"orderung der wissenschaftlichen Forschung" (FWF).

\section{Filtered manifolds and weighted jet bundles}\label{sectionweighted}

In this section we introduce the notion of a filtered manifold and discuss the concept of weighted jet bundles over filtered manifolds as it was introduced by Morimoto, see especially \cite{Morimoto3}.

\subsection{Filtered manifolds}\label{filteredmanifold}

We start by collecting some basic facts about filtered manifolds.

\begin{defin} 
A \textit{filtered manifold} is a smooth manifold $M$ together with a filtration of its tangent bundle $TM$
by vector subbundles $\{T^iM\}_{i\in\Z}$ such that:
\begin{itemize}
\item $T^iM\supseteq T^{i+1}M$
\item $T^0M=0$ and there exists $\ell\in\N$ with $T^{-\ell}M=TM$
\item for sections $\xi\in\Gamma(T^iM)$ and $\eta\in\Gamma(T^jM)$ the Lie bracket $[\xi,\eta]$ is a section of $T^{i+j}M$
\end{itemize}
\end{defin}

By the first two properties, the tangential filtration can be written as
$$TM=T^{-k}M\supsetneq T^{-k+1}M\supsetneq...\supsetneq T^{-1}M$$ for some
$k\in\N$ with $T^iM=TM$ for $i\leq -k$ and $T^iM=0$ for $i\geq 0$.
The number $k\in \N$ is called the \textit{depth} of the filtered manifold. 
In the sequel we will always suppose that a filtered manifold is given in this form.
\\Given a filtered manifold $M$ of depth $k$, one can form the \textit{associated graded vector bundle} $\textrm{gr}(TM)$ 
to the filtered vector bundle $TM$, which is defined as
$$\textrm{gr}(TM)=\bigoplus_{i\in\Z}\textrm{gr}_{i}(TM)=\bigoplus_{i={-k}}^{-1}\textrm{gr}_{i}(TM),$$
where $\textrm{gr}_{i}(TM)=T^{i}M/T^{i+1}M$. 
\\Now consider the operator $\Gamma(T^iM)\times\Gamma(T^jM)\rightarrow \Gamma(\textrm{gr}_{i+j}(TM))$
given by
$(\xi,\eta)\mapsto q_{i+j}([\xi,\eta])$, where $q_{i+j}: T^{i+j}M\rightarrow\textrm{gr}_{i+j}(TM)$ is the natural projection. One verifies directly that this operator is bilinear over smooth
functions and therefore it is induced by a bilinear bundle map $T^iM\times
T^jM\rightarrow \textrm{gr}_{i+j}(TM)$. Moreover, it obviously
factorises to a bundle map
$\textrm{gr}_i(TM)\times\textrm{gr}_j(TM)\rightarrow\textrm{gr}_{i+j}(TM),$
since for $\xi\in\Gamma(T^{i+1}M)$ or $\eta \in\Gamma(T^{j+1}M)$  we have
$[\xi,\eta]\in\Gamma(T^{i+j+1}M)$. Hence we obtain a bilinear vector bundle map on the associated graded bundle
$$\mathcal L :\textrm{gr}(TM)\times\textrm{gr}(TM)\rightarrow\textrm{gr}(TM),$$
which makes the fiber
$\textrm{gr}(T_xM)$ over $x\in M$ into a nilpotent graded Lie algebra:
$$\mathcal L_x(\gr_i(T_xM),\gr_j(T_xM))\subset \gr_{i+j}(T_xM).$$ 

\begin{defin}
Let $M$ be a filtered manifold.
\begin{enumerate}
\item The tensorial bracket $\mathcal L :\textrm{gr}(TM)\times\textrm{gr}(TM)\rightarrow\textrm{gr}(TM)$ induced from the Lie bracket of vector fields on
$\gr(TM)$ is called the \textit{Levi bracket}.
\item The nilpotent graded Lie algebra $(\gr(T_xM),\mathcal L_x)$ is called the \textit{symbol algebra} of the
filtered manifold $M$ at the point $x\in M$.
\end{enumerate}
\end{defin}

Suppose $M$ and $N$ are filtered manifold and let $f:M\rightarrow N$ be a \textit{local isomorphism of filtered manifolds}, 
i.e. a local diffeomorphism, whose tangent map $Tf$ satisfies $Tf(T^iM)=T^iN$ for all $i\in \Z$.
Then for each point $x\in M$ the tangent map $T_xf$ at $x$ induces an
isomorphism of graded vector spaces between $\textrm{gr}(T_xM)$ and
$\textrm{gr}(T_{f(x)}N)$ and the compatibility of the pullback of vector fields with the Lie
bracket easily implies that this actually is an isomorphism of
graded Lie algebras. Therefore the symbol algebras are basic invariants one can associate to a filtered manifold.
\\The symbol algebra of a filtered manifold may change from point to point and so $\gr(TM)$ needs not to be locally trivial as a vector bundle of Lie algebras.
Note that, if $\n=\n_{-k}\oplus...\oplus \n_{-1}$ is a nilpotent graded Lie algebra and $M$ a filtered manifold such that the symbol algebra in each point is isomorphic to $\n$,
one has a natural notion of a \textit{frame bundle of the associated graded bundle}
$\mathcal P(\textrm{gr}(TM))$. Denoting by $\mathcal{ P}_x(\textrm{gr}(TM))$ the
space of all graded Lie algebra isomorphisms $\mathfrak
n\rightarrow \textrm{gr}(T_xM)$, the frame bundle is defined by the
disjoint union 
\begin{equation}\label{frame}
\mathcal P(\textrm{gr}(TM)):=\sqcup_{x\in
M}\mathcal{ P}_x(\textrm{gr}(TM)).
\end{equation}  
The bundle $\mathcal
P(\textrm{gr}(TM))$ is a principal bundle with structure group
$\textrm{Aut}_{\textrm{gr}}(\mathfrak n)$, the group of all Lie
algebra automorphisms of $\mathfrak n$ that preserve
the grading. 

\begin{rem}
Any ordinary smooth manifold can be seen as a trivial filtered manifold $TM=T^{-1}M$. 
The symbol algebra $\gr(T_xM)$ at $x\in M$ is then just the tangent space $T_xM$ viewed as an abelian Lie algebra.
\end{rem}

One of the best studied examples of a non-trivial filtered manifold is a contact manifold: 

\begin{ex}\label{Contactex1}
A \textit{contact manifold} is a manifold $M$ of dimension $2n+1$ together with a vector subbundle $H\subset TM$ of rank $2n$ such that in each point $x\in M$ the Levi bracket
$\mathcal L_x: H_x\times H_x\rightarrow T_xM/H_x$ is non-degenerate. Hence a contact manifold $(M,H)$ of dimension $2n+1$ is a filtered manifold $H=:T^{-1}M\subset T^{-2}M=TM$, whose symbol algebra in each point is isomorphic to the Heisenberg Lie algebra of dimension $2n+1$. 
\end{ex}

In Section \ref{RegInfFlag} we will see several other interesting examples of filtered manifolds.

\subsection{The weighted order of a differential operator}\label{WeightedOrder}

As indicated in the introduction we now adapt the notion of order for differential operators on filtered manifolds  with respect to the filtration of the tangent bundle.

\begin{defin} \label{defweighted}
Let $M$ be a filtered manifold.

\begin{enumerate}

\item {A local vector field $\xi$ of $M$ is of weighted order $\leq r$, if $\xi$ is a local section of $T^{-r}M$.
The smallest number $r\in\N_0$ such that this holds is called the \textit{weighted order} $\textrm{ord}(\xi)$ of $\xi$.}

\item {A linear differential operator $D:C^{\infty}(M,\mathbb C)\rightarrow C^{\infty}(M,\mathbb C)$ on $M$
is of weighted order $\leq r$, if for each point $x\in M$ there exists a local frame $\{X_1,...,X_n\}$ of $TM$ defined on an open neighbourhood $U\subset M$ of $x$  such that
$$D|_U=\sum_{\alpha\in\N_0^n} a_\alpha X_1^{\alpha_1}...X_n^{\alpha_n}\quad\textrm{ with }\quad a_\alpha\in C^\infty(U,\C)$$
where for all non zero terms in this sum $\sum_{i=1}^n \alpha_i\textrm{ord}(X_i)\leq r$.
The smallest number $r\in\N_0$  such that this holds is called the \textit{weighted order} of $D$. }

\end{enumerate}

\end{defin}

Suppose that $M$ is a filtered manifold of depth $k$. Choosing an open subset of $M$, over which all subbundles $T^{-\ell}M$ of the tangent bundle trivialise, one can always construct  an \textit{adapted local frame} of $TM$, i.e. a local frame 

\begin{equation}\label{adaptedframe}
\{X_{1,1},..., X_{1, i(1)},...,X_{k,1},...,X_{k, i(k)}\}
\end{equation}
such that  $\{X_{1,1},..., X_{1, i(1)},...,X_{\ell,1},...,X_{\ell, i(\ell)}\}$ is a local frame of $T^{-\ell}M$ for all $\ell \leq k$. Using that the filtration of the tangent bundle is compatible with the Lie bracket of vector fields and that vector fields act as derivations on the algebra of smooth functions, one shows directly that the following proposition holds, for a proof see \cite{thesis}.

\begin{prop}
 A linear differential operator $D$ on $M$ is of weighted order $r$ if and only if the following two conditions hold:
\begin{enumerate}
\item for each point $x\in M$ there exists an adapted local frame of $TM$ defined on some open neighbourhood $U_x$ of $x$ such that
 \begin{equation}\label{D}
 D|_{U_x}=\sum_{|\alpha|\leq r} a_{\alpha} X_{1,1}^{\alpha_{1,1}}...X_{k, i(k)}^{\alpha_{k,i(k)}} \quad\textrm{ with }\quad a_\alpha\in C^\infty(U_x,\C),
 \end{equation}
where $\alpha=(\alpha_{1,1},...,\alpha_{k, i(k)})\in\N_0^n$ is a multi-index with
$|\alpha|:=\sum_{j=1}^k \sum_{\ell=1}^{i(j)}j\alpha_{j,\ell}$ 
\item there exists at least one point $x_0\in M$ such that the local description (\ref{D}) of $D|_{U_{x_0}}$ satisfies that $a_\alpha(x_0)\neq 0$ for some $\alpha$ with $|\alpha|=r$. 
\end{enumerate}
If $D$ is a linear differential operator of weighted order $r$, then $D$ is for any choice of local adapted frames of $TM$ of this form.
\end{prop}

\subsection{The universal enveloping algebra of a nilpotent graded Lie algebra}\label{Univalg}

Suppose that $\n=\n_{-k}\oplus...\oplus\n_{-1}$ is a graded nilpotent Lie algebra and denote by $[$ , $]$ its Lie bracket. Recall that the universal enveloping algebra $\U(\n)$ of the Lie algebra $\n$ is defined as the following quotient $$\U(\n):=\T(\n)/\I,$$
where $\T(\n)$ is the tensor algebra of $\n$ and $\I$ the two-sided ideal generated by elements of the form $X\otimes Y-Y\otimes X-\mathcal [X, Y]$ for $X,Y\in \n$.
The grading on $\n$ induces an algebra grading on the tensor algebra given by

$$
\T(\n)=\bigoplus_{i=0}^\infty\T_{-i}(\n)
$$

$$\T_{-i}(\n)=\{\sum_j X_{j_1}\otimes...\otimes X_{j_{s(j)}}: X_{j_\ell}\in\n_{j_\ell}\textrm{ and } \sum_{\ell=1}^{s(j)} j_\ell=-i \quad\forall j\}.$$

Since $\n$ is a graded Lie algebra, the ideal $\I$ is homogeneous and hence the grading passes to an algebra grading on the universal enveloping algebra

\begin{equation}\label{weightedgrad2}
\mathcal{U}(\n)=\bigoplus_{i=0}^\infty \mathcal{U}_{-i}(\n).
\end{equation}

Denote by $\mathcal S(\n)$ the symmetric algebra of $\n$ and by $\mathfrak S_p$ the symmetric group on $p$ letters. Consider the linear map $\Psi:\mathcal S(\n)\rightarrow\U(\n)$ given by

\begin{equation}\label{symmetrisation1}
\Psi(X_1...X_p)=\frac{1}{p!}\sum_{\theta\in\mathfrak S_p}X_{\theta(1)}...X_{\theta(p)} \quad\quad X_1,...,X_p\in\n,
\end{equation}
where the product on the left side is taken in $\mathcal S(\n)$ and on the right side in $\U(\n)$. From the Poincar\'e-Birkhoff-Witt Theorem one deduces that this is a linear isomorphism, called the \textit{symmetrisation}, see e.g. \cite{Dixmier}. 
Denoting by $\Psi_{-j}$ the restriction of $\Psi$ to the symmetric algebra $\mathcal S(\n_{-j})$, one shows easily that also the map 

\begin{equation}\label{symmetrisation2}
\mathcal S(\n_{-1})\otimes...\otimes\mathcal S(\n_{-k})\rightarrow \U(\n) 
\end{equation}

\begin{equation*}
x_{-1}\otimes...\otimes x_{-k}\mapsto \Psi_{-1}(x_{-1})...\Psi_{-k}(x_{-k})
\end{equation*}

is a linear isomorphism. Note that this map restricts to a linear isomorphism 

\begin{equation}\label{symmetrisation3}
\mathcal S_{-i}(\n):=\bigoplus_{1i_1+...+ki_k=i}S^{i_1}\n_{-1}\otimes...\otimes S^{i_k}\n_{-k}\cong\U_{-i}(\n),
\end{equation}
where $S^{i_j}\n_{-j}$ denotes the symmetric tensors of degree $i_j$ of $\n_{-j}$.
\\In particular, for a filtered manifold $M$ the vector space given by the $-i$-th grading component $\U_{-i}(\gr(T_xM))$ of the universal enveloping algebra of the symbol algebra $(\gr(T_xM),\mathcal L_x)$ at $x\in M$ is by (\ref{symmetrisation3}) always isomorphic to $\mathcal S_{-i}(\gr(T_xM))$ and we therefore conclude that $$\U_{-i}(\gr(TM)):=\bigsqcup_{x\in M}\U_{-i}(\gr(T_xM))$$ can be natural endowed with the structure of a vector bundle over $M$.

\subsection{Weighted jet bundles}\label{WeightedJet}

Suppose that $\pi: E\rightarrow M$ is a complex or real vector bundle over $M$ and let $\Gamma_x(E)$ be the space of germs of smooth sections of $E$ at the point $x\in M$. 
For $r\in\N_0$ two sections $s,s'\in\Gamma_x(E)$ are called $r$-\textit{equivalent} $\sim_r$,
if  $$D(\langle\lambda,s-s'\rangle)(x)=0$$
for all linear differential operators $D$ on $M$ of weighted order $\leq r$ and all sections $\lambda$ of the dual bundle $E^*$, where
$\langle$  ,  $\rangle: \Gamma(E^*)\times\Gamma(E)\rightarrow C^\infty(M,\mathbb C)$ is the evaluation.

\begin{defin}
The quotient of $\Gamma_x(E)$ by the equivalence relation $\sim_r$ 
$$\mathcal J_x^r(E):=\Gamma_x(E)/\sim_r$$ 
is called \textit {the space of jets of weighted order} $r$ \textit{with source} $x\in M$. 
For $s\in\Gamma_x(E)$ we denote by $j^r_xs$ the class of $s$ in $\mathcal J^r_x(E)$.
\end{defin}

Since for $s<r$ the relation $\sim_s$ is coarser than the relation $\sim_r$, we have linear projections
$\pi^r_s: \J^r_x(E)\rightarrow\J^s_x(E)$  for  $s<r$. Moreover, we have the following proposition, see also \cite{Morimoto3}.

\begin{prop}\label{jetexact}
For $x\in M$ and $r\in\N$ we have an exact sequence of vector spaces
\begin{equation*}
\xymatrix{0\ar[r] & \mathcal{U}_{-r}(\ggr(T_xM))^*\otimes E_x\ar[r]^{\quad\quad\iota}& \mathcal J^r_x(E)\ar[r]^{\pi^r_{r-1}}& \mathcal J^{r-1}_x(E)\ar[r]& 0  \\ }
\end{equation*}
\end{prop}

\begin{proof}
Suppose that $s$ is a local section defined around $x$ with $j_x^{r-1}s=0$. Choosing some local trivialisation of $E$ over an open neighbourhood $U$ of $x$, we can view $s$ as a smooth function $(s_1,...,s_m): U\subseteq M\rightarrow \R^m$. For vector fields  $\xi_1,..,\xi_\ell\in\Gamma(TM)$ with $\sum_i \textrm{ord}(\xi_i)=r$ the value
$(\xi_1\cdot...\cdot\xi_\ell\cdot s)(x)\in \R^m$ depends only on the values of the vector fields at the point $x$, since $j^{r-1}_xs=0$.
By the same reason, it actually just depends on the elements $q_{-ord(\xi_i)}(\xi_i(x))\in\gr_{-ord(\xi_i)}(T_xM)$. Therefore we obtain a well defined linear map
\begin{equation}\label{T} 
\mathcal T_{-r}(\textrm{gr}(T_xM))\rightarrow \R^m.
\end{equation}
Additionally we have the symmetries of differentiation, like for example $$\xi_1\cdot\xi_2\cdot...\cdot\xi_\ell\cdot s-\xi_2\cdot\xi_1\cdot...\cdot\xi_\ell\cdot s=[\xi_1,\xi_2]\cdot...\cdot\xi_\ell\cdot s$$ and since $q_{-(ord(\xi_1)+ ord(\xi_2))}([\xi_1,\xi_2](x))=\mathcal L_x(\xi_1(x),\xi_2(x)),$ the map (\ref{T}) factorises to a linear map $$\mathcal{U}_{-r}(\textrm{gr}(T_xM))\rightarrow \R^m.$$ Hence any element in the kernel of the projection $\J^r_x(E)\rightarrow \J^{r-1}_x(E)$ defines an element in
$\mathcal{U}_{-r}(\textrm{gr}(T_xM))^*\otimes E_x$ and so we have a linear map $\tau: \ker(\pi_{r-1}^r)\rightarrow\mathcal{U}_{-r}(\textrm{gr}(T_xM))^*\otimes E_x$, which obviously is injective. 
\\To see that $\tau$ is also surjective we construct an inverse map. Let $U$ be an open neighbourhood of $x$ over which all filtration components of the tangent bundle and $E$ trivialise. Choose an adapted local frame $\{X_{1,1},...,X_{k,i(k)}\}$ of $TM$ defined on $U$. Note that such an adapted local frame defines an isomorphism $T_yM\cong\gr(T_yM)$ for all
$y\in U$, where the vector space spanned by $\{X_{j,1}(y),...,X_{j,i(j)}(y)\}$ is mapped onto $\gr_{-j}(T_yM)$.
Now suppose $\{f_{1,1},...,f_{k,i(k)}\}$ are smooth functions which vanish at $x\in M$ and satisfy that $(X_{j,p}\cdot f_{\ell,q})(x)=\delta_{\ell j}\delta_{qp}$.
\\By (\ref{symmetrisation3}) the monomials $X_{1,1}^{\alpha_{1,1}}(x)...X_{k,i(k)}^{\alpha_{k,i(k)}}(x)$ with $|\alpha|=r$ form a basis of $\U_{-r}(\gr(T_xM))$.
For each multi-index $\alpha$ with $|\alpha|=r$ define $\phi_\alpha\in\U_{-r}(\gr(T_xM))^*$ as the linear functional given by
\begin{align*}
\phi_\alpha(X_{1,1}^{\alpha_{1,1}}(x)...X_{k,i(k)}^{\alpha_{k,i(k)}}(x))= & X_{1,1}^{\alpha_{1,1}}\cdot...\cdot X_{k,i(k)}^{\alpha_{k,i(k)}}(f_{1,1}^{\alpha_{1,1}}...f_{k,i(k)}^{\alpha_{k,i(k)}})(x)\\
\phi_\alpha((X_{1,1}^{\beta_{1,1}}(x)...X_{k,i(k)}^{\beta_{k,i(k)}}(x)))= & 0 \quad\textrm{ for } \beta\neq\alpha.
\end{align*}
By its construction the functionals $\{\phi_\alpha: |\alpha|=r\}$ form a basis of $\U_{-r}(\gr(T_xM))^*$ and we define a linear map 
\begin{equation*}
\iota: \U_{-r}(\gr(T_xM))^*\otimes E_x\rightarrow \J^r_x(E)
\end{equation*}
\begin{equation*}
\phi_\alpha\otimes e \mapsto j^r_x(f_{1,1}^{\alpha_{1,1}}...f_{k,i(k)}^{\alpha_{k,i(k)}}s), 
\end{equation*}
where $s$ is some section of $E$ with $s(x)=e\in E_x$.
It is easy to see that $\iota$ is a well defined injection with values in $\ker(\pi^r_{r-1})$, which is inverse to $\tau$. 
\end{proof}

\begin{defin}
For $r\in\N_0$ the disjoint union over all $x$ of $\J^r_x(E)$  $$\J^r(E):=\bigsqcup_{x\in M}\J_x^r(E)$$
is called the \textit{space of jets of weighted order} $r$. We denote by $\pi^r: \J^r(E)\rightarrow M$ the natural projection.
\end{defin}

Using Proposition \ref{jetexact} and an adapted local frame for $\gr(TM)$, it can be easily shown that any vector bundle chart of $E$ gives rise to a local trivialisation of $\J^r(E)$ and one can endow $\J^r(E)$ with the unique structure of a smooth manifold such that these trivialisations are smooth and $\pi^r: \J^r(E)\rightarrow M$ is a vector bundle. Finally, one obtains the following theorem, see \cite{Morimoto3} and for a proof \cite{thesis}.

\begin{thm}\label{Jetvec}
Let $\pi:E\rightarrow M$ be a vector bundle over a filtered manifold $M$.
\begin{enumerate}
\item{For $r\in\N_0$ the natural projection $\pi^r: \J^r(E)\rightarrow M$ is a vector bundle with fiber $\J^r_x(E)$ isomorphic to $\bigoplus_{i=0}^r \mathcal U_{-i}(\textrm{\emph{gr}}(T_xM))^*\otimes E_x$.}
\item{For  $r>s$  the projections $$\pi^r_s: \J^r(E)\rightarrow \J^s(E)$$ are vector bundle homomorphisms and for $r\in\N$ we have an exact sequence of vector bundles
\begin{equation*}
\xymatrix{0\ar[r] & \mathcal{U}_{-r}(\ggr(TM))^*\otimes E\ar[r]^{\quad\quad\iota}& \mathcal J^r(E)\ar[r]^{\pi^r_{r-1}}& \mathcal J^{r-1}(E)\ar[r]& 0  \\ }
\end{equation*}
}
\end{enumerate}
\end{thm}

\subsection{The weighted symbol of a differential operator}

Suppose that $E$ and $F$ are vector bundles over a filtered manifold $M$.

\begin{defin}
A differential operator $D:\Gamma(E)\rightarrow\Gamma(F)$ is of \textit{weighted order} $\leq r$, if for any point $x\in M$ and any two sections $s,t\in\Gamma(E)$ the equation
$j^r_xs=j^r_xt$ implies that $D(s)(x)=D(t)(x)$. The smallest number $r\in\N_0$ such that this holds, is called the \textit{weighted order} of $D$.
\end{defin}

Given a differential operator $D:\Gamma(E)\rightarrow \Gamma(F)$ of weighted order $r$, we obtain a bundle map $\phi: \J^r(E)\rightarrow F$ defined by $\phi(j^r_xs)=D(s)(x)$.
Conversely, if $\phi: \J^r(E)\rightarrow F$ is a bundle map, then $D=\phi\circ j^r$ defines a differential operator of weighted order at most $r$, where $j^r:\Gamma(E)\rightarrow\Gamma(\J^r(E))$ is the \textit{universal differential operator of weighted order} $r$ given by
$s\mapsto(x\mapsto j^r_xs)$. Therefore we can view a differential operator $D:\Gamma(E)\rightarrow\Gamma(F)$ of weighted order $r$  equivalently as a bundle map $\J^r(E)\rightarrow F$. Note that a differential operator $D:\Gamma(E)\rightarrow\Gamma(F)$ of weighted order $r$ is linear if and only if the associated bundle map $\phi: \J^r(E)\rightarrow F$ is a vector bundle map.

\begin{defin} 
Let $D: \Gamma(E)\rightarrow\Gamma(F)$ be a differential operator of weighted order $r$ with associated bundle map $\phi: \J^r(E)\rightarrow F$.
The \textit{weighted symbol} $\sigma_r(\phi)$ of $D$ is the composition of $\phi$ with the canonical inclusion $\iota: \mathcal{U}_{-r}(\textrm{gr}(TM))^*\otimes E\hookrightarrow \J^r(E).$ 

\begin{equation}
\xymatrix{
0\ar[r]& \mathcal{U}_{-r}(\textrm{gr}(TM))^*\otimes E \ar[r]^{\quad\quad\iota}\ar[dr]^{\sigma_r(\phi)}& \J^r(E) \ar[r]^{\pi^r_{r-1}}\ar[d]^{\phi}& \J^{r-1}(E)\ar[r]& 0\\ & & F& &\\}
\end{equation}

Sometimes we will also just write $\sigma(\phi)$ for the weighted symbol.

\end{defin}

\begin{rem}
If $M$ is a trivial filtered manifold $TM=T^{-1}M$, then the weighted jet bundle $\J^r(E)$ of a vector bundle $E$ coincides with the usual vector bundle $J^r(E)$ of jets of order $r$.
The symbol algebra of $M$ at some point $x\in M$ is just the tangent space $T_xM$ viewed as abelian Lie algebra and the weighted symbol $\sigma_r(\phi)$ of a differential operator $\phi: \J^r(E)\rightarrow F$ is the usual principal symbol $\sigma_r(\phi): S^r(T_xM)^*\otimes E\rightarrow F$.
\end{rem}

\section{Regular infinitesimal flag structures}\label{SectionRegInfFlag}

In this section we explain briefly the notion of a regular infinitesimal flag structure and give some examples, for a detailed discussion of these structures see \cite{CSbook}.

\subsection{Parabolic subalgebras of semisimple Lie algebras}

Suppose that $\mathfrak g$ is a complex semisimple Lie algebra. 

\begin{defin}
A $\textit{Borel subalgebra}$ of $\g$ is a maximal solvable subalgebra of $\g$. A subalgebra
$\mathfrak p$ of $\g$ is called a $\textit{parabolic}$ subalgebra, if $\mathfrak p$ contains a Borel subalgebra.
\end{defin}

Let $\mathfrak h$ be a Cartan subalgebra of $\mathfrak g$. Then we denote by $\Delta$ the set of roots associated to $\mathfrak h$ and for $\alpha\in\Delta$ we write $\g_\alpha$ for the corresponding root space. Further, choose a simple subsystem of roots $\Delta^0\subset\Delta$ and denote by $\Delta^+$ the corresponding system of positive roots.
The subalgebra of $\g$ defined by $$\mathfrak b=\mathfrak h\oplus\bigoplus_{\alpha\in\Delta^+}\mathfrak g_{\alpha}$$ is a maximal solvable subalgebra of $\g$, called the $\textit{standard Borel subalgebra}$ associated to $\mathfrak h$ and $\Delta^0$. A $\textit{standard parabolic subalgebra}$ is a subalgebra which contains $\mathfrak b$.
It is well known that standard parabolic subalgebras can be classified by subsets $\Sigma\subset\Delta^0$ of simple roots.
In fact, the map $$\mathfrak p \mapsto \Sigma_{\mathfrak p}=\{\alpha\in\Delta^0:\mathfrak g_{-\alpha}\nsubseteq \mathfrak p \}$$ defines a bijection between standard parabolic subalgebras and subsets of simple roots, where the inverse is given by assigning to a subset $\Sigma$ of simple roots the algebra $\p_{\Sigma}$, which is the direct sum of $\b$ and all root spaces corresponding to negative roots, which can be written as a linear combination of elements of $\Delta^0\setminus\Sigma$.
\\The fact that Cartan subalgebras and the choice of a simple subsystem of roots are unique up to conjugation implies that every parabolic subalgebra is conjugate by an inner automorphism of $\g$ to a standard one. Hence up to conjugation a parabolic subalgebra can be uniquely described by a subset of simple roots. 
\\There is an alternative description of parabolic subalgebras as subalgebras in semisimple Lie algebras that determine $|k|$-gradings on semisimple Lie algebras.

\begin{defin}
Let $\mathfrak g$ be a complex or real semisimple Lie algebra and $k>1$ an integer.
A \textit{$|k|$-grading} on $\mathfrak g$ is a vector space decomposition
$$\mathfrak g=\mathfrak g_{-k}\oplus...\oplus\mathfrak g_0\oplus...\oplus\mathfrak g_k$$
such that
\begin{itemize}
\item{ $[\mathfrak g_i,\mathfrak g_j]\subseteq\mathfrak g_{i+j}$, where we set $\mathfrak g_i=\{0\}$ for $|i|>k$}
\item{the subalgebra $\mathfrak g_{-}:=\mathfrak g_{-k}\oplus...\oplus\mathfrak g_{-1}$ is generated  as Lie algebra by $\mathfrak g_{-1}$}
\item{$\mathfrak g_{\pm k}\neq \{0\}$}
\end{itemize}
\end{defin}

By the grading property each $\g_i$ is a $\g_0$-module and it is easy to see that the Killing form of $\g$ induces a duality between the $\g_0$-modules $\g_i$ and $\g_{-i}$.
\\Let $\mathfrak p$ be a standard  parabolic subalgebra of a complex semisimple Lie algebra $\g$ and denote by $ht_{\Sigma_p}(\alpha)$ the $\Sigma_{\mathfrak p}$-height of $\alpha$, i.e.  the sum of all coefficients of elements in $\Sigma_{\mathfrak p}$ in the representation of $\alpha$ as linear combination of simple roots.
Then $\mathfrak p$ determines a $|k|$-grading on $\mathfrak g$ as follows
$$\mathfrak g_0:=\mathfrak h\oplus\bigoplus_{ht_{\Sigma_{\mathfrak p}}(\alpha)=0}\mathfrak g_{\alpha}
\quad\quad\textrm{ and }\quad\quad\mathfrak g_i:=\bigoplus_{ht_{\Sigma_{\mathfrak p}}(\alpha)=i}\mathfrak g_{\alpha}$$
where $\mathfrak g^0:=\g_0\oplus \g_1\oplus...\oplus \g_k$ equals $\mathfrak p$.
\\Conversely, given a $|k|$-grading on a complex semisimple Lie algebra $\mathfrak g=\mathfrak g_{-k}\oplus...\oplus\mathfrak g_k$, one can show that there is a choice of $\mathfrak h $ and $\Delta^0$ such that $\mathfrak g^0$ contains the standard Borel subalgebra and hence it is parabolic. Moreover, the grading is then given by
$\Sigma_{\mathfrak g^0}$-height. Therefore one obtains a bijection between conjugation classes of parabolic subalgebras and isomorphism classes of $|k|$-gradings, for details we refer to \cite{CSbook}.
Moreover, using the description of $|k|$-gradings in terms of weights, one shows, see e.g. \cite{CSbook}:

\begin{prop} \label{g_0}
Suppose that $\g=\g_{-k}\oplus...\oplus\g_k$ is a complex semisimple $|k|$-graded Lie algebra. Then the following holds:
\begin{enumerate}
\item{The subalgebra $\g_0$ is reductive. It is called the \textit{Levi subalgebra} of the parabolic subalgebra $\p:=\g^0$.}
\item{Let $\h$ be a Cartan subalgebra of $\g$ and $\Delta^0$ a simple subsystem of roots such that $\p$ is a standard parabolic subalgebra. Denote by
$\{H_\alpha\}_{\alpha\in\Delta^0}$ the basis of $\h$, where $H_{\alpha}\in\h$ corresponds under the isomorphism $\h\rightarrow \h^*$ induced by the Killing form to $\alpha\in\Delta^0$. Then we have $$\g_0=\z(\g_0)\oplus\h_0\oplus\bigoplus_{ht_{\Sigma_{\mathfrak p}}(\alpha)=0}\mathfrak g_{\alpha},$$
where $\z(\g_0)$ is the center of $\g_0$ and $\h_0$ is the linear span of all the $H_\alpha$ with $\alpha\in\Delta^0\backslash \Sigma_\p$. Hence the dimension of $\z(\g_0)$ equals the number of elements in $\Sigma_\p$. Moreover, the subalgebra $\h_0$ is a Cartan subalgebra of the semisimple part $\g_0^{ss}$ of $\g_0$ whose corresponding root decomposition is exactly $$\g_0^{ss}=\h_0\oplus\bigoplus_{ht_{\Sigma_{\mathfrak p}}(\alpha)=0}\mathfrak g_{\alpha}.$$
 }
\end{enumerate}
\end{prop}

\begin{rem}[The real case]
A subalgebra $\p$ of a real semisimple Lie algebra $\g$ is called \textit{parabolic}, if its complexification $\p^\C$ is a parabolic subalgebra in the complexification $\g^\C$ of $\g$. From the discussion in the complex case it follows therefore immediately that the subalgebra $\g^0$ of a real $|k|$-graded semisimple Lie algebra $\g$ is a parabolic subalgebra.
It is well known that parabolic subalgebras of real semisimple Lie algebras can be described up to conjugation by subsets of simple restricted roots respectively by subsets $\Sigma$ of non-compact simple roots of $\g^\C$ satisfying that if in the Satake diagram of $\g$ two roots are connected by an arrow they are either both in $\Sigma$ or none of them is.
This description can then be used analogously as in the complex case to establish a correspondence between parabolic subalgebras and $|k|$-gradings on real semisimple Lie algebras, details can for instance be found in \cite{CSbook}.
\end{rem}

\subsection{Regular infinitesimal flag manifolds}\label{RegInfFlag}

Let $\g=\g_{-k}\oplus...\oplus \g_0\oplus...\oplus\g_k$ be a $|k|$-graded semisimple Lie algebra and set 
$$\g_-:=\g_{-k}\oplus...\oplus \g_{-1}\quad\textrm{ and }\quad\p_+:=\g_1\oplus...\oplus \g_k.$$
Suppose that $G$ is a
Lie group with Lie algebra $\g$.  A closed subgroup $P\subseteq G$ is a \textit{parabolic subgroup} corresponding to the given $|k|$-grading, if its Lie algebra equals the parabolic subalgebra $\p:=\g_0\oplus\p_+$.
Having fixed a parabolic subgroup $P$ corresponding to the grading, the \textit{Levi subgroup} $G_0$ of $P$ is the closed subgroup of $P$ given by 
$$G_0:=\{g\in P: Ad(g)(\mathfrak g_{i})\subset\mathfrak g_{i} \textrm{ for } i=-k,...,k\},$$
where $Ad: G\rightarrow GL(\g)$ denotes the adjoint representation of $G$. Its Lie algebra is the reductive Lie algebra $\g_0$.
Moreover, the definition of $G_0$ shows that the adjoint action induces a group homomorphism 
\begin{equation}\label{grouphom}
Ad: G_0 \rightarrow Aut_{\gr}(\g_-), 
\end{equation}
where $Aut_{\gr}(\g_-)$ denotes as in Section \ref{filteredmanifold} the group of grading preserving Lie algebra automorphisms of $\g_-$. \\For a semisimple Lie group $G$, whose Lie algebra is endowed with a $|k|$-grading and a parabolic subgroup $P$ with Lie algebra $\p$ a \textit{regular infinitesimal flag structure of type} $(G,P)$ on a manifold $M$ consists of the following data:
\begin{itemize}
\item{a filtration of the tangent bundle $TM=T^{-k}M\supset...\supset T^{-1}M,$ which makes $M$ into a filtered manifold whose symbol algebra in each point is isomorphic to the Lie algebra $\g_-$.  }
\item{a reduction $\mathcal G_0\rightarrow M$ of the structure group of the frame bundle $\mathcal P(\gr(TM))$ of $\gr(TM)$ to the group $G_0$ with respect to the homomorphism (\ref{grouphom}).}
\end{itemize}

Let us give some overview of examples, for an extensive discussion of these examples and many others see \cite{CSbook}:

\begin{ex}
Suppose that $\g=\g_{-1}\oplus\g_0\oplus\g_1$ is a $|1|$-graded semisimple Lie algebra, $G$ a Lie group with Lie algebra $\g$ and $P$ a parabolic subgroup corresponding to the given grading.  A regular infinitesimal flag structure of type $(G,P)$ is just a reduction of the structure group of the frame bundle of $M$ to the Levi subgroup $G_0$ via $Ad:G_0\rightarrow GL(\g_{-1})$. So it is just a first order $G_0$-structure. Among these geometric structures, which are called \textit{almost hermitian symmetric structures}, we have conformal structures, almost quaternionic structures and almost Grassmannian structures.
\end{ex}

\begin{ex}
A regular infinitesimal flag structure consists of a filtration of the tangent bundle and a reduction of the structure group corresponding to the homomorphism (\ref{grouphom}). If this homomorphism is an isomorphism, the regular infinitesimal flag structure is just a filtered manifold with symbol algebra $\g_-$. 
The most prominent examples of this type are generic rank $2$ distributions on five dimensional manifolds, studied by Cartan in \cite{Cartan}, generic rank $3$ distributions on six dimensional manifolds, investigated by Bryant in \cite{Bryant} as well as quaternionic contact structures introduced by Biquard in \cite{Biquard}.
\end{ex}

\begin{ex}\label{excontact}
A \textit{contact grading} on $\g$ is a $|2|$-grading $\g=\g_{-2}\oplus\g_{-1}\oplus\g_0\oplus\g_1\oplus\g_2$ such that $\g_-$ is a Heisenberg Lie algebra. It turns out that such gradings exists only on simple Lie algebras $\g$ and are unique up to isomorphism, for a complete classification see \cite{Yam}. Given a real contact grading on a simple Lie algebra $\g$ and corresponding groups $P\subset G$, a regular infinitesimal flag structure of type $(G,P)$ can be interpreted as a contact structure $TM=T^{-2}M\supset T^{-1}M$  together with a reduction of the structure group of the frame bundle $\mathcal P(T^{-1}M)$ of the contact subbundle to $G_0$, since $Aut_{\gr}(\g_-)$ can be seen as a subgroup of $GL(\g_{-1})$. 
Geometric structures of this form are for instance oriented contact structures, partially integrable almost CR-structures of hypersurface type and Lagrangean contact structures.
\end{ex}

In the sequel we will always assume that we are dealing with a manifold endowed with a regular infinitesimal flag structure of some type $(G,P)$, where we have not only a reduction of the structure group to $G_0$ via (\ref{grouphom}), but also a further reduction $\G_0^{ss}\rightarrow M$ to the semisimple part $G_0^{ss}$ of $G_0$. In the case of conformal structures, this additional reduction means to choose a metric from the conformal class or in the case of contact structures it can be interpreted as the choice of a contact from. This assumption is not necessary, but, since we are not interested in questions of invariance here, it is even something natural to do and it will allow us to formulate the results of this work in a more approachable way.

\subsection{Natural vector bundles and representations of $G_0^{ss}$}\label{naturalvec}

Suppose that $M$ is a manifold endowed with a geometric structure $(\G_0^{ss}, \{T^iM\})$ as in Section \ref{RegInfFlag}. For a representation $\E$ of $G_0^{ss}$, we will denote by $$E:=\G_0^{ss}\times_{G_0^{ss}}\E$$ the vector bundle associated to the principal bundle $\G_0^{ss}$ with standard fiber $\E$. Note that any $G_0^{ss}$-equivariant map between two $G_0^{ss}$-representations $\E$ and $\F$ induces a vector bundle homomorphism between the corresponding associated vector bundles $E$ and $F$.
\\The reduction of the structure group of $\mathcal P(\gr(TM))$ to $G_0^{ss}$ obviously induces isomorphisms of vector bundles 
$$\G_0^{ss}\times_{G_0^{ss}}\g_{-i}\cong\gr_{-i}(TM)\quad\textrm{ and }\quad\mathcal G_0^{ss}\times_{G_0^{ss}}\g_-\cong  \gr(TM).$$
Since the Killing form induces a duality of $G_0$-modules between $\g_{i}$ and $\g_{-i}$, we have also the following isomorphisms of vector bundles
$$\G_0^{ss}\times_{G_0^{ss}}\g_i\cong\gr_{-i}(TM)^*\quad\textrm{ and }\quad \G_0^{ss}\times_{G_0^{ss}}\p_+\cong \gr(TM)^*.$$ 
Later we will need the decomposition of the $G_0^{ss}$-representation $\g_{-1}$ into irreducibles. Using the description of the grading on $\g$ in terms of roots and Proposition \ref{g_0} one verifies directly that the following holds: 

\begin{lem}\label{g_1}
Suppose that $\g=\g_{-k}\oplus...\oplus\g_k$ is complex semisimple $|k|$-graded Lie algebra.
Let $\h\subset\g$ be a Cartan subalgebra and $\Delta^0=\{\alpha_1,...\alpha_n\}$ a simple subsystem of roots  such that the parabolic subalgebra $\p$ is standard. We define $J$ as the subset of $I:=\{1,...,n\}$ consisting of those elements $i\in I$ with $\alpha_i\in\Sigma_\p$. 
Then the $\g_0^{ss}$-module $\g_{-1}$ decomposes into irreducibles as follows $$\g_{-1}=\bigoplus_{j\in J}\g_{-1,j},$$ where $\g_{-1,j}$ is the unique irreducible representation with highest weight $-\alpha_{j}|_{\h_0}$.
\end{lem}

We finish this section by fixing some notation.
Suppose that $\E$ and $\F$ are complex irreducible representations of $\g_0^{ss}$ with highest weights $\lambda$ and $\mu$ respectively. It is well known that there exists an irreducible component $\E\circledcirc\F$ of multiplicity one in $\E\otimes\F$, which has highest weight $\lambda+\mu$. It is called the \textit{Cartan product} of $\E$ and $\F$ and up to multiplication by a scalar there is a unique projection $\E\otimes\F\rightarrow\E\circledcirc\F$.
\\If $\g$ is real and $\E$ is a real irreducible representation of $\g_0^{ss}$ having no complex invariant structure, then the complexification $\E^{\C}$ is a complex irreducible representation of $\g_0^{ss}$ and we will mean by the highest weight of $\E$ the highest weight of $\E^\C$. Given two such real irreducible representation $\E$ and $\F$ of $\g_0^{ss}$ we denote by $\E\circledcirc\F$ the unique irreducible component in $\E\otimes\F$, whose complexification equals $\E^\C\circledcirc\F^\C$ in $\E^\C\otimes\F^\C$. 

\begin{rem}\label{remweight}
Let $\g$ be a complex semisimple $|k|$-graded Lie algebra. For a simple root $\alpha_i$ of $\g$ denote by $\omega_{\alpha_i}\in\h^*$ the corresponding fundamental weight, which is characterised by $2\frac{<\omega_{\alpha_i},\alpha_\ell>}{<\alpha_\ell,\alpha_\ell>}=\delta_{i\ell}$ where $< , >$ is the Killing form of $\g$. Recall that the highest weight of an irreducible representations of $\g$ can be uniquely written as a linear combination $\sum_{i\in I}a_i\omega_{\alpha_i}$ with $a_i\in\N_0$.
Note that by (2) of Proposition \ref{g_0} the highest weights of irreducible representations of $\g_0^{ss}$ can be uniquely written as restrictions to $\h_0$ of linear combinations $\sum_{i\in I\backslash J}a_i\omega_{\alpha_i}$ with $a_i\in \N_0$. 
\end{rem}

\section{Lie algebra cohomology}\label{LieAlgCoh}

One of the crucial ingredients, we will use in the sequel, is Lie algebra cohomology and especially Kostant's version of the Bott-Borel-Weil Theorem (\cite{Kost}). In this section we therefore review  Kostant's results as far as we need them and fix some notation related to them.   
\\Suppose that $\g=\g_-\oplus\g_0\oplus\p_+$ is a $|k|$-graded semisimple Lie algebra, $G$ a Lie group with Lie algebra $\g$ and $P$ a parabolic subgroup corresponding to the given $|k|$-grading. Moreover, let $\V$ be a representation of $G$. 
The Lie algebra cohomology, in which we will be interested, is the cohomology $H^*(\g_-,\V)$ of $\g_-$ with values in $\V$, where $\V$ is viewed as representation of $\g_-$. Let us recall how the cochain complex for computing this Lie algebra cohomology looks like. 
The $n$-th cochain space is given by the space $\Lambda^n\g_-^*\otimes\V$ of $n$-linear alternating maps from $\g_-$ to $\V$ and the differential $\partial: \Lambda^n\g_-^*\otimes\V\rightarrow \Lambda^{n+1}\g_-^*\otimes\V$ is defined by
$$\partial(f)(X_0,...,X_n):=\sum_{i=0}^n (-1)^iX_if(X_0,...,\hat{X^i},...,X_n)$$
$$+\sum_{i<j}(-1)^{i+j}f([X_i,X_j],X_0,...,\hat{X^i},...,\hat{X^j},...,X^n)$$
for $X_0,...,X_n\in\mathfrak g_-$, where the hat over an argument denotes omission.
\\Note that the Levi subgroup $G_0$ acts (by its definition) on $\g_-$ via the adjoint action as well as on $\V$ by restriction. Hence we also have an induced action of $G_0$ on the cochain spaces $\Lambda^n\g_-^*\otimes\V$ and it can be directly verified that the differentials $\partial$ are $G_0$-equivariant. Therefore the cohomology spaces $H^n(\g_-,\V)$ of this complex are naturally $G_0$-modules. 
\\Dualising the differential of the cochain complex for computing the cohomology of $\p_+$ with values in the dual representation $\V^*$ leads to a $P$-equivariant map
$\partial^*:\Lambda^{n+1}\p_+\otimes\V\rightarrow\Lambda^n\p_+\otimes\V$ which satisfies $\partial^*\circ\partial^*=0$. 
Explicitly, 
$\partial^*$ is given by
$$\partial^*(Z_0\wedge...\wedge Z_n\otimes v)=\sum_{i=0}^n (-1)^{i+1} Z_0\wedge....\wedge\hat{Z_i}\wedge...\wedge Z_n\otimes Z_n v$$
$$+\sum_{i<j}(-1)^{(i+j)} [Z_i,Z_j]\wedge Z_0\wedge...\wedge\hat{Z_i}\wedge...\wedge\hat{Z_j}\wedge...\wedge Z_n\otimes v.$$

Since the Killing form induces an isomorphism $\g_-^*\cong\p_+$ of $G_0$-modules, we may identify the $G_0$-modules $\Lambda^n\mathfrak g_-^*\otimes \mathbb V$ and $\Lambda^n\mathfrak p_+\otimes \mathbb V$ and view the codifferential $\partial^*$ as a map 
$$\partial^*: \Lambda^n\mathfrak g_-^*\otimes \mathbb V\rightarrow \Lambda^{n-1}\mathfrak g_-^*\otimes \mathbb V.$$
In \cite{Kost} Kostant showed that the operators $\partial$ and $\partial^*$ are adjoint for some inner product (hermitian in the complex case) on the spaces
$\Lambda^n\mathfrak g_-^*\otimes \mathbb V$ and that one has an \textit{algebraic Hodge decomposition}: 
\begin{equation}\label{Hodge}
\Lambda^n\mathfrak g_-^*\otimes \mathbb V=\textrm{im}(\partial^*)\oplus\ker(\square)\oplus\textrm{im}(\partial),
\end{equation}
where $\square$ is the $G_0$-equivariant map $\partial\partial^*+\partial^*\partial$ on $\Lambda^n\mathfrak g_-^*\otimes \mathbb V$.Moreover, one has $\ker(\partial^*)=\textrm{im}(\partial^*)\oplus\ker(\square)$ and $\ker(\partial)=\textrm{im}(\partial)\oplus\ker(\square).$
This gives rise to $G_0$-module isomorphisms 
\begin{equation}\label{isocohom}
H^n(\g_-,\V)\cong\ker(\partial^*)/\im(\partial^*)\quad\textrm{ and }\quad H^n(\mathfrak g_-,\mathbb V)\cong \ker(\square).
\end{equation}
Note that via the second isomorphism $H^n(\mathfrak g_-,\mathbb V)$ can be naturally viewed as $G_0$-submodule in $\Lambda^n\mathfrak g_-^*\otimes \mathbb V$.
Since $\partial^*$ is obtained by dualising the differential for computing the cohomology $H^*(\p_+,\V^*)$, we see that $H^n(\g_-,\V)\cong H^n(\p_+,\V^*)^*$ as $G_0$-modules.
Using the Hodge decomposition (\ref{Hodge}), Kostant described $H^n(\g_-,\V)$ respectively $H^n(\p_+,\V^*)$ for an irreducible representation $\V$ of $G$ explicitly as $G_0$-module.
We will be interested only in the cohomology in degree zero and one. There Kostant's result \cite{Kost} reads as follows:

\begin{thm}\label{Kostant}
Let $\g=\g_-\oplus\g_0\oplus\p_+$ be a complex $|k|$-graded semisimple Lie algebra.
Suppose that $\V$ is a complex irreducible representation of $\g$ with highest weight $\lambda$. For a root $\alpha\in\Delta$ denote the corresponding root reflection by $s_{\alpha}$ and let $\rho$ be the lowest form of $\g$ given by the sum of the fundamental weights of $\g$. Then we have:
\begin{enumerate}
\item{$H^0(\p_+,\V)=\V^{\p_+}:=\{v\in\V:Xv=0 \quad\forall X\in\p_+ \}$. As $\g_0^{ss}$-module $\V^{\p_+}$ is the $\g_0^{ss}$-representation with highest weight $\lambda|_{\h_{0}}$.}
\item{As $\g_0^{ss}$-module the first cohomology space is isomorphic to the following direct sum 
$$H^1(\p_+,\V)\cong\bigoplus_{\alpha\in\Sigma_\p}\F_{\lambda_{s_\alpha}},$$
where $\F_{\lambda_{s_\alpha}}$ denotes the irreducible representation of $\g_0^{ss}$ with highest weight $s_\alpha(\lambda+\rho)-\rho$ restricted to $\h_0$. Moreover, viewing $H^1(\p_+,\V)$ as submodule in $\p_+^*\otimes\V\cong\g_-\otimes\V$, each irreducible component $\F_{\lambda_{s_\alpha}}$ has multiplicity one in $\g_-\otimes\V$ and a highest weight vector of $\F_{\lambda_{s_\alpha}}$ is given by the tensor product of a nonzero element of the root space $\g_{-\alpha}$ and a nonzero weight vector of weight $s_\alpha(\lambda)$.}
\end{enumerate}
\end{thm}
 
From the Hodge decomposition (\ref{Hodge}) one also deduces that by restricting $\partial$ to $\im(\partial^*)$ respectively $\partial^*$ to $\im(\partial)$ one obtains isomorphisms
$$\partial: \im(\partial^*)\cong\im(\partial) \quad\textrm{ and }\quad \partial^*: \im(\partial)\cong\im(\partial^*).$$
In general, these two maps are not inverse to each other. However, we may define for later purposes the map $$\delta^*:\Lambda^n\g_-^*\otimes \V\rightarrow\Lambda^{n-1}\g_-^*\otimes \V,$$
which is the inverse of $\partial$ on $\im(\partial)$ and zero on $\ker(\partial^*)$. Obviously, we have again $\delta^*\circ\delta^*=0$ and so $\delta^*$ is a differential. Since, by construction, $\delta^*$ differs from
$\partial^*$ on $\im(\partial)$ just by a $G_0$-equivariant isomorphism of $\im(\partial)$, we conclude that $\delta^*$ is as well $G_0$-equivariant. Moreover, it defines the same Hodge decomposition.

\section{Prolongation of overdetermined systems on regular infinitesimal flag manifolds}

In this section we shall now study a large class of semi-linear differential operators between natural vector bundles over regular infinitesimal flag manifolds, which give rise to overdetermined systems. Given such a semi-linear differential operator $D$, we will establish a conceptual method to rewrite the semi-linear system $Ds=0$ as a system of partial differential equations of the form $\tilde\nabla\Sigma+C(\Sigma)=0$, where $\tilde\nabla$ is a linear connection on some vector bundle $V$ over the regular infinitesimal flag manifold $M$ and $C: V\rightarrow T^*M\otimes V$ is a bundle map. 

\subsection{Semi-linear systems on regular infinitesimal flag manifolds corresponding to $|k|$-gradings such that $\z(\g_0)$ is one dimensional}\label{onedim}

Suppose that $\g=\g_-\oplus\g_0\oplus\p_+$ is a $|k|$-graded semisimple Lie algebra where $\z(\g_0)$ is one dimensional. Let $G$ be a simply connected Lie group with Lie algebra $\g$, $P$ a parabolic subgroup with Lie algebra $\p$. Moreover, suppose that $M$ is a manifold endowed with a geometric structure $(\{T^iM\}, \G_0^{ss})$ of type $(G,P)$ as in Section \ref{RegInfFlag}.

\begin{rem}~
\begin{enumerate}
\item{We assume $G$ to be simply connected only to ensure that a representation of $\g$ integrates to a representation of $G$. This will allow us to formulate the result of this section in a uniform way. The condition can be dropped, when one is dealing with some particular representation and group, where one knows that this is the case.}
\item{Concerning the examples mentioned in Section \ref{RegInfFlag}, the condition on the center just excludes the case of partially integrable almost CR-structures of hypersurface type and its real analogue the Lagrangean contact structures  }
\end{enumerate}
\end{rem}

Let now $\E$ be an irreducible representation of $G_0^{ss}$ and $r>0$ some integer. From Proposition \ref{g_0} and Lemma \ref{g_1} we deduce that $\g_{-1}$ is an irreducible representation of $\g_0^{ss}$ and we set $\F:=\circledcirc^r\g_{-1}^*\circledcirc\E$.
For $ g\in G_0$ the graded Lie algebra automorphsim $Ad(g):\g_-\rightarrow\g_-$ lifts by the universal property of the universal enveloping algebra to a graded algebra automorphism $\U(\g_-)\rightarrow\U(\g_-)$. In particular, $\U_{-r}(\g_-)$ can be given the structure of a $G_0$-module and one shows directly that the isomorphism (\ref{symmetrisation3}) between $\mathcal S_{-r}(\g_-)$ and $\U_{-r}(\g_-)$ is $G_0$-equivariant. In particular, we have a $G_0^{ss}$-equivariant
 linear projection $\U_{-r}(\g_-)^*\otimes \E\rightarrow S^r\g_{-1}^*\otimes \E.$ Composing this projection with the projection
$S^r\g_{-1}^*\otimes \E\rightarrow\circledcirc^r \g_{-1}^*\circledcirc\E$, we obtain a $G_0^{ss}$-equivariant linear projection 
$\U_{-r}(\g_-)^*\otimes \E\rightarrow\circledcirc^r \g_{-1}^*\circledcirc\E$ and hence also a corresponding surjective vector bundle map
\begin{equation}\label{naturalprojection}
\U_{-r}(\gr(TM))^*\otimes E\rightarrow F=\circledcirc^r \gr_{-1}(TM)^*\circledcirc E.
\end{equation}
For semi-linear differential operators $D: \Gamma(E)\rightarrow\Gamma(F)$ of weighted order $r$ with weighted symbol given by (\ref{naturalprojection}), i.e. $D$ is of the form $D=D_1+D_2$, where $D_1$ is a linear differential operator of weighted order $r$ with weighted symbol given by (\ref{naturalprojection}) and $D_2$ is a differential operator of weighted order at most $r-1$, we will prove the following theorem:

\begin{thm} \label{main1}
Suppose that $\E$ is a irreducible representation of $G_0^{ss}$ and let $r>0$ be some integer. Set $\F=\circledcirc^r \g_{-1}^*\circledcirc\E$. Then there exists:
\\(1) a natural graded vector bundle
$$V=V_0\oplus...\oplus V_N$$  over $M$ with $V_0=E$
\\(2) for any choice of a principal $G_0^{ss}$-connection $\nabla$ on $\mathcal G_0^{ss}\rightarrow M$ and for any choice of a splitting of the filtration of the tangent
bundle (i.e. an isomorphism $TM\cong\ggr(TM)$ that restricts to a map $T^iM\rightarrow \bigoplus_{j\geq i}\ggr_j(TM)$ and the component in $\ggr_i(TM)$ equals the image of the projection $T^{i}M\rightarrow T^iM/T^{i+1}M$):
\begin{enumerate}
\item[$\bullet$] a linear connection $\widetilde{\nabla}$ on $V$ 
\item[$\bullet$] a linear differential operator
$L: \Gamma(V_0)\rightarrow \Gamma(V)$ of weighted order $N$ satisfying that $p_0(L(s))=s$, where $p_0$ denotes the projection $V\rightarrow V_0=E$,
\end{enumerate}
with the following property:
\\For every semi-linear differential operator $D:\Gamma(E)\rightarrow\Gamma(F)$ of weighted order $r$ with symbol given by the projection (\ref{naturalprojection})
$$\sigma(D):\mathcal U_{-r}(\ggr(TM))^*\otimes E\rightarrow F=\circledcirc^r \ggr_{-1}(TM)^*\circledcirc E.$$
the linear differential operator $L$ induces a bijection between the following sets
$$\{s\in\Gamma(E): D(s)=0\}\leftrightarrow\{\Sigma\in\Gamma(V):(\widetilde{\nabla}+C)(\Sigma)=0\}$$
for some bundle map $C:V\rightarrow T^*M\otimes V$. The inverse is induced by the projection $p_0: V\rightarrow V_0=E$.
\end{thm}

To prove this theorem we proceed in three steps:

\subsection*{1. Step - The construction of the graded vector bundle $V$}

From Theorem \ref{Kostant} one immediately deduces the following proposition:

\begin{prop}\label{V}
Suppose that $\g$ is a $|k|$-graded semisimple Lie algebra such that the center of $\g_0$ is one dimensional. Let $\E$ be an irreducible representation of $\g_0^{ss}$ and $r>0$ an integer. Then there exists an irreducible representation $\V=\V[\E,r]$ of $\g$ such that as $\g_0^{ss}$-modules we have $$H^0(\g_-,\V)\cong\E\quad\textrm{ and }\quad H^1(\g_-,\V)\cong\circledcirc^r \g_{-1}^*\circledcirc\E.$$ 
\end{prop}
  
\begin{proof}
If the restriction of $\lambda=\sum_{i\in I\backslash J}a_i \omega_{\alpha_i}$ to $\h_0$  is the highest weight of the dual representation $\E^*$ (see remark \ref{remweight} for the notation), then define $\V=\V[\E,r]$ to be the irreducible representation of $\g$, whose dual representation has highest weight $\mu:=\lambda+(r-1)\omega_{\alpha_j}\in\h^*$, where $\omega_{\alpha_j}$ is the fundamental weight corresponding to the simple root  $\alpha_j$ in $\Sigma_\p$. From Theorem \ref{Kostant} we deduce that $H^0(\p_+,\V^*)\cong\E^*$ and that $H^1(\p_+,\V^*)$ is the irreducible representation of $\g_0^{ss}$ with highest weight $s_{\alpha_{j}}(\mu+\rho)-\rho$ restricted to $\h_0$, which equals the restriction of $\lambda-r\alpha_j$ to $\h_0$. From Lemma \ref{g_1} we therefore conclude that $H^1(\p_+,\V^*)\cong\circledcirc^r \g_{-1}\circledcirc\E^*$. Since $H^*(\g_-,\V)\cong H^*(\p_+.\V^*)^*$ as $\g_0^{ss}$-modules, the claim follows.
\end{proof}

Suppose that we have fixed $\E$ and $r>1$ and let $\V=\V[\E,r]$ be the irreducible representation of $\g$ from Proposition \ref{V}. There always exists a unique element $e\in\z(\g_0)\subset\g$, whose adjoint action represents the grading on $\g$, i.e.
$[e,X]=jX \textrm{ for } X\in\g_j,$ see \cite{Yam} or \cite{CSbook}.
In particular, it acts diagonalisably on $\g$ and therefore on any finite dimensional representation of $\g$.  So we can decompose $\V$ into eigenspaces for the action of the grading element $e$ on $\V$. Observe that for an eigenvector $v$ with eigenvalue $c$ and $X\in\g_j$ the vector $X\cdot v$ is eigenvector with eigenvalue $c+j$, since $e\cdot X\cdot v=X\cdot e\cdot v+[e,X]\cdot v$. Therefore, denoting by $c$ the eigenvalue with smallest real part, it follows from the irreducibility of $\V$ that the set of eigenvalues is given by $\{c,c+1,...,c+N-1\}$ for some $N\geq 1$. For $0\leq i\leq N$ let $\V_i$ be the eigenspace to the eigenvalue $c+i$ and set $\V_i=0$ for $i<0$ or $i>N$. Then we obtain a decomposition of $\V$ as follows:

\begin{equation}\label{decomV1}
\V=\V_0\oplus...\oplus\V_N \quad\textrm{ such that }\quad \g_i\cdot\V_j\subseteq\V_{i+j} \quad\textrm{ for all } i, j\in\Z
\end{equation}

In particular, each subspace $\V_i$ is invariant under the action of $\g_0$ respectively under the action of $G_0$.
\\Moreover, the gradings on $\g_-$ and $\V$ induce a grading on the space $\Lambda^n\g_-^*\otimes\V$, where the $i$-th grading component is given by
\begin{equation}\label{gradLambda}
(\Lambda^n\g_-^*\otimes\V)_i=\bigoplus_{t=n}^{nk}(\Lambda^n_{-t}\g_-)^*\otimes\V_{i-t},
\end{equation}
with $$\Lambda^n_{-t}\g_-=\bigoplus_{i_1+...+i_n=-t}\g_{i_1}\wedge...\wedge\g_{i_n}.$$
It follows immediately from (\ref{decomV1}) that the Lie algebra  differential $\partial$ is grading preserving. 
We will denote the restriction of $\partial$ to the $i$-th grading component by $\partial_i$.
\\By (\ref{decomV1}) the $G_0^{ss}$-invariant subspace $\V_0\subset\V$ is contained in $\ker(\partial)$. Since $\ker(\partial)=H^0(\g_-,\V)$ is an irreducible representation of $G_0^{ss}$ we conclude that as $G_0^{ss}$-modules
\begin{equation}\label{decomV3}
\V_0=\ker(\partial)=H^0(\g_-,\V)\cong\E
\end{equation}

In particular, we see that $\partial_i: \V_i\rightarrow  \bigoplus_{t=1}^k\g_{-t}^*\otimes\V_{i-t}$ is injective for $i>0$.
\\Now consider the Hodge decomposition (\ref{Hodge}) of Section \ref{LieAlgCoh}:
$$\g_-^*\otimes\V=\im(\partial)\oplus\ker(\square)\oplus\im(\delta^*)=\ker(\partial)\oplus\im(\delta^*).$$ Since $\delta^*$ is also obviously compatible with the gradings on the cochain spaces, the same holds for $\square$.
Therefore we obtain that $$(\g_-^*\otimes\V)_i=\im(\partial_i)\oplus\ker(\square_i)\oplus\im(\delta^*_i)=\ker(\partial_i)\oplus\im(\delta^*_i).$$
Since $H^1(\g_-,\V)\cong\ker(\square)$, the first cohomolgy $H^1(\g_-,\V)$ may be viewed as a $G_0^{ss}$-submodule of  $\g_-^*\otimes\V$.  
We know from Theorem \ref{Kostant} that $\circledcirc^{r}\g_{-1}^*\circledcirc\mathbb E\cong H^1(\g_-,\V)$ has multiplicity one in $\g_-^*\otimes\V$. Using Theorem \ref{Kostant} we can even determine the grading component, in which $\circledcirc^{r}\g_{-1}^*\circledcirc\mathbb E$ is lying.
In fact, by Theorem \ref{Kostant} a highest weight vector of the irreducible representation $\circledcirc^{r}\g_{-1}\circledcirc\mathbb E^*\cong H^1(\p_+,\V^*)$ viewed as a submodule in $\p_+^*\otimes\V^*\cong \g_-\otimes\V^*$ is of the form $X\otimes v$, where $X\in\g_{-\alpha_j}$ and $v\in\V^*$ is a weight vector of weight $s_{\alpha_j}(\mu)=\mu-(r-1)\alpha_j$, where $\mu$ is the highest weight of $\V^*$. It can be easily seen that $\V^*_{\ell}$ consists of all those weight spaces of $\V^*$ corresponding to weights of the form $\mu-\ell \alpha_j-\sum_{i\in I\setminus \{j\}} n_i\alpha_i$, where $n_i\in\N_0$ and $I$ is the index set of the simple roots of $\g$ as in Section \ref{naturalvec}. This implies that the irreducible component $\circledcirc^{r}\g_{-1}\circledcirc\mathbb E^*$ lies in $\g_{-1}\otimes \V_{r-1}^*$. Since $H^1(\p_+,\V^*)^*$ is isomorphic to $H^1(\g_-,\V)$, we obtain that
$$\circledcirc^{r}\g_{-1}^*\circledcirc\mathbb E\cong\ker(\square)=\ker(\square_r)\subset \g_{-1}^*\otimes\V_{r-1}.$$
In particular, for $0<i< r$ we therefore have the following exact sequence
\begin{equation}\label{H1V}
\xymatrix{
0\ar[r] & \V_i\ar[r]^{\partial_i\quad\quad\quad} &  \bigoplus_{s=1}^k\g_{-s}^*\otimes\V_{i-s}\ar[r]^{\partial_i\quad}& \bigoplus_{t=2}^{2k}(\Lambda^2_{-t}\g_-)^*\otimes\V_{i-t}.\\}
\end{equation}

\begin{prop}\label{U_iV_0}
For $0\leq i\leq N$ there exist $G_0$-equivariant inclusions $$\phi_i: \V_i\hookrightarrow\U_{-i}(\g_-)^*\otimes\V_0.$$
For $i<r$ these inclusions are even isomorphisms $\phi_i: \V_i\cong\U_{-i}(\g_-)^*\otimes\V_0.$
\end{prop}

\begin{proof}
By means of restriction we can view $\V$ as a representation of $\g_-$ or equivalently as a $\U(\g_-)$-module. From (\ref{decomV1}) we conclude that $$\U_{-i}(\g_-)\V_i\subseteq\V_0 \quad\quad\textrm{ for all } 0\leq i\leq N.$$ Now we define $\phi_i$ by 
$$\phi_i: \V_i\rightarrow \U_{-i}(\g_-)^*\otimes\V_0$$ 
$$v\mapsto(u\mapsto-u^{\top}v),$$
where $u\mapsto u^{\top}$ is the unique anti-automorphism of $\U(\g_-)$ such that $X^{\top}$=-X for $X\in\g_-$, see e.g. \cite{Dixmier}. It satisfies $(X_1X_2...X_n)^{\top}=(-1)^nX_nX_{n-1}...X_1$ for $X_1,...,X_n\in\g_-$. 
Observing that
\begin{equation}\label{induktiv}
\U_{-i}(\g_-) = \bigoplus_{j=1}^k\g_{-j}\otimes\U_{-(i-j)}(\g_-)/\mathcal J_i
\end{equation}
\begin{equation*}
\mathcal J_i = <X\otimes Yu-Y\otimes Xu-[X,Y]\otimes u: X\in\g_{-p},Y\in\g_{-q}, u\in\U_{-(i-p-q)}(\g_-)>
\end{equation*}
we can prove by induction on $i$ that all $\phi_i$ are injective. 
\\For $i=0$ the result holds, since $\phi_0=-id$.
\\The map $\phi_1: \V_1\rightarrowÊ \g_{-1}^*\otimes \V_0$ equals $\partial_{1}: \V_1\rightarrow \g_{-1}^*\otimes \V_0$, which is injective by (\ref{decomV3}) and so the result holds also for $i=1$ .
\\Now suppose that $\phi_j$ is injective for all $j<i$ and consider the following commutative diagram:
\begin{equation*}
\xymatrix{ \V_i \ar[r]^{\partial_i\quad\quad}\ar[d]_{id} & \bigoplus_{s=1}^k \g_{-s}^*\otimes\V_{i-s}\ar[r]^{\partial_i}\ar[d]_{\imath} &
\bigoplus_{t=2}^{2k} (\Lambda^2_{-t}\g_-)^*\otimes\V_{i-t}\ar[d]_{\jmath}\\
Ê\V_i\ar[r]^{\tilde\partial_i\quad\quad\quad\quad\quad\quad} & \bigoplus_{s=1}^k (\g_{-s}\otimes\U_{-(i-s)}(\g_-))^*\otimes \V_0 \ar[r]^{\tilde\partial_i\quad} & \bigoplus_{t=2}^{2k} (\Lambda^2_{-t}\g_-\otimes\U_{-(i-t)}(\g_-))^*\otimes\V_0\\ }
\end{equation*} 
where
\begin{align}
&\imath(f)(X\otimes u)=u^{\top}f(X)=-\phi_{i-s}(f(X))(u)\nonumber\\
&\qquad \textrm{ for } X\otimes u\in \g_{-s}\otimes\U_{-(i-s)}(\g_-)\nonumber\\
&\jmath(g)(X\wedge Y\otimes u)=u^{\top}g(X\wedge Y)=-\phi_{i-t}(g(X\wedge Y))(u)\nonumber\\
&\qquad\textrm{ for } X\wedge Y\otimes u\in \Lambda^2_{-t}\g_-\otimes\U_{-(i-t)}(\g_-)\nonumber\\
&\partial_i(v)(X)=Xv\quad \textrm{ for } X\in\mathfrak g_{-s}\nonumber\\
&\partial_i(f)(X\wedge Y)=Xf(Y)-Yf(X)-f([X,Y])\nonumber\\
&\qquad \textrm{ for } X\wedge Y\in \Lambda^2_{-t}\g_-\nonumber\\
&\tilde\partial_i(h)(X\wedge Y\otimes u)=h(X\otimes Yu)-h(Y\otimes Xu)-h([X,Y]\otimes u) \nonumber\\
&\qquad\textrm{ for } X\wedge Y\otimes u\in \Lambda^2_{-t}\g_-\otimes\U_{-(i-t)}(\g_-).\nonumber
\end{align}

Since $\partial\circ\partial=0$, the commutativity of the diagram implies that the composition $\imath\circ\partial_i$ has values in the kernel of $\tilde\partial_i$. By (\ref{induktiv}) the kernel $\ker(\tilde\partial_i)$ coincides with $\U_{-i}(\g_-)^*\otimes\V_0\subset\bigoplus_{s=1}^k(\g_{-s}\otimes\U_{-(i-s)}(\g_-))^*\otimes \V_0$ and so we have 
$$\tilde\partial_i=\imath\circ\partial_i: \V_i\rightarrow \U_{-i}(\g_-)^*\otimes\V_0.$$
Moreover, since $(\imath\circ\partial_i)(v)(X\otimes u)=u^{\top}(Xv)=-(Xu)^{\top}v$, wee see that $\imath\circ\partial_i=\phi_i.$
We know by (\ref{decomV3}) that $\partial_i: \V_i\rightarrowÊ \bigoplus_{s=1}^k\g_{-s}^*\otimes\V_{i-s}$ is injective and by induction hypothesis also $\imath$ is injective. Therefore we have that
\begin{equation*}
\phi_i: \V_i\stackrel{\partial_i}{\cong}\im(\partial_i)\stackrel{\imath}{\hookrightarrow}\U_{-i}(\g_-)^*\otimes\V_0
\end{equation*}
is injective and so the first assertion holds.
\\Since for $0< i<r$ we have by (\ref{H1V}) an exact sequence
\begin{equation*}
\xymatrix{
0\ar[r] & \V_i\ar[r]^{\partial_i\quad\quad\quad} &Ê \bigoplus_{s=1}^k\g_{-s}^*\otimes\V_{i-s}\ar[r]^{\partial_i\quad}&
\bigoplus_{t=2}^{2k}(\Lambda^2_{-t}\g_-)^*\otimes\V_{i-t},\\}
\end{equation*}
it follows by induction from the commutative diagram above that
\begin{equation*}
\phi_i: \V_i\stackrel{\partial_i}{\cong}\ker(\partial_i)\stackrel{\imath}{\cong} \U_{-i}(\g_-)^*\otimes \V_0
\end{equation*}
is an isomorphism for $0< i<r$.
\end{proof}

Since $\phi_i$ is an isomorphism for $i<r$, one deduces from the commutative diagram of the proof of Proposition \ref{U_iV_0} for $i=r$ that
$$\V_r\stackrel{\partial_r}{\cong}\im(\partial_r)\subset \ker(\partial_r)=\im(\partial_r)\oplus\ker(\square_r)\stackrel{\imath}{\cong}\ker(\tilde\partial_r)=\U_{-r}(\g_-)^*\otimes\V_0.$$
Note that the map $\imath$ viewed as a map $$\imath :\bigoplus_{s=1}^k\g_{-s}^*\otimes\V_{r-s}\rightarrow\bigoplus_{s=1}^k\g_{-s}^*\otimes\U_{-(r-s)}(\g_-)^*\otimes\V_0$$
can be written $$\imath=\sum_{s=1}^k-id\otimes\phi_{r-s}$$ and therefore the isomorphism induced by $\imath$ between $\ker(\square_r)=\ker(\square)$ and $\circledcirc^{r}\g_{-1}^*\circledcirc\mathbb E$
is given by

\begin{equation}\label{Isocohom}
\ker(\square)\hookrightarrow\g_{-1}^*\otimes\V_{r-1}\stackrel{-id\otimes\phi_{r-1}}{\cong}\g_{-1}^*\otimes\U_{r-1}(\g_-)^*\otimes\E\rightarrow \g_{-1}^*\otimes S^{r-1}\g_{-1}^*\otimes\E\rightarrow \circledcirc^{r}\g_{-1}^*\circledcirc\mathbb E.
\end{equation}

We conclude that we obtain a $G_0^{ss}$-equivariant isomorphism
$$\phi_r: \V_r \cong(\U_{-r}(\g_-)^*\otimes\E)\cap \K,$$
where  $\K\subset\U_{-r}(\g_-)^*\otimes\E $ denotes the kernel of the $G_0^{ss}$-equivariant projection
$\U_{-r}(\g_-)^*\otimes\E\rightarrow \circledcirc^{r}\g_{-1}^*\circledcirc\mathbb E$.
Since $\ker(\square)=\ker(\square_r)$, it follows by induction as in the proof of the Proposition \ref{U_iV_0} that  we have $G_0^{ss}$-equivariant isomorphisms
$$\phi_i: \V_i \cong(\U_{-i}(\g_-)^*\otimes\E)\cap (\U_{-(i-r)}(\g_-)^*\otimes \K)\quad\textrm{ for } \quad i\geq r.$$
Now we define $V$ as the graded vector bundle associated to $\V$:
$$V=V_0\oplus...\oplus V_N=\mathcal G_0^{ss}\times_{G_0^{ss}}\V_{0}\oplus...\oplus\V_N,$$
where $V_0=E$. Moreover, we define $K:=\G_0\times_{G_0^{ss}}\K$ as the natural vector bundle corresponding to $\K$.
Since the isomorphisms $\phi_i$ are $G_0^{ss}$-equivariant, they induce vector bundle isomorphisms between the corresponding vector bundles
\begin{align}
&\phi_i: V_i \cong\mathcal U_{-i}(\gr(TM))^*\otimes E&\textrm{ for all } i<r\nonumber\\
&\phi_i: V_i \cong(\mathcal U_{-i}(\gr(TM))^*\otimes E)\cap (\mathcal U_{-(i-r)}(\gr(TM))^*\otimes K)&\textrm{ for all } i\geq r.\nonumber
\end{align}

\begin{rem}
Since for $i\geq r$ we have $\mathcal U_{-i}(\gr(TM))^*\otimes E\cap \mathcal U_{-(i-r)}(\gr(TM))^*\otimes K\cong V_i$, we see that a linear differential operator $D:\Gamma(E)\rightarrow\Gamma(F)$ of weighted order $r$ with weighted symbol given by (\ref{naturalprojection}) is of weighted finite type in the sense of \cite{Mor5} and \cite{N}.
\end{rem}

\subsection*{2. Step - The construction of the connection $\tilde{\nabla}$ and the differential operator $L$}

Since the maps $\partial$ and $\delta^*$ are $G_0^{ss}$-equivariant and compatible with the grading on $\Lambda^n\g_-^*\otimes\V$, they give rise to grading preserving vector bundle maps 
$$\partial:\Lambda^n\gr(TM)^*\otimes V\rightarrow \Lambda^{n+1}\gr(TM)^*\otimes V$$
$$\delta^*:\Lambda^{n}\gr(TM)^*\otimes V\rightarrow\Lambda^{n-1}\gr(TM)^*\otimes V,$$
where the grading on the vector bundle $\Lambda^{n}\gr(TM)^*\otimes V$ is induced by the grading on $\Lambda^n\g_-^*\otimes\V$.
\\Let us now choose a principal connection on $\mathcal G_0^{ss}\rightarrow M$. Then we get induced linear connection on all associated vector
bundles and we will denote all of them by $\nabla$. In particular, we obtain a linear connection $\nabla: \Gamma(V)\rightarrow \Gamma(T^*M\otimes V)$ on $V$. The filtrations of $V$ and $TM$ induce a filtration of $T^*M\otimes V$, where the $\ell$-th filtration component $(T^*M\otimes V)^\ell$ consists of all elements in $T^*M\otimes V$ of homogeneity $\geq \ell$, i.e. $\phi\in (T^*M\otimes V)^\ell\textrm{ if and only if } \phi(T^iM)\subset V^{i+\ell} \textrm{ for } i<0.$ 
Since $\nabla$ is induced from a principal $G_0^{ss}$-connection, it has to preserve the grading on $V$. Hence it raises homogeneity by one: $$\nabla: \Gamma(V^i)\rightarrow\Gamma((T^*M\otimes V)^{i+1}).$$
Choosing a splitting of the filtration of the tangent bundle $TM\cong \gr(TM)$, we can view $\partial$ and $\delta^*$ as grading respectively filtration preserving vector bundle maps on $\Lambda^nT^*M\otimes V$.
In particular, the following definition makes sense:
$$\widetilde{\nabla}:=\nabla+\partial:\Gamma(V)\rightarrow \Gamma(T^*M\otimes V).$$
$\tilde\nabla$ is a linear connection on $V=\ker(\partial)\oplus\im(\delta^*)=V_0\oplus\im(\delta^*)$, which is of homogeneity $\geq 0$ and whose lowest homogeneous component is given by the algebraic operator $\partial$. Now consider the following linear differential operator $$\delta^*\circ\tilde\nabla: \Gamma(V)\rightarrow \Gamma(\textrm{im}(\delta^*))\subseteq \Gamma(V).$$ It is of homogeneity $\geq 0$ with lowest homogeneous component given by $\delta^*\circ\partial$. If we restrict this operator to $\textrm{im}(\delta^*)$, the lowest component $\delta^*\circ\partial$ is the identity on $\textrm{im}(\delta^*)$ and $-(\delta^*\tilde\nabla-id)$ is (at most) $N$-step nilpotent. Therefore $\delta^*\tilde\nabla$ is invertible on $\Gamma(\textrm{im}(\delta^*))$ with inverse given by the von Neumann series $$(\delta^*\tilde\nabla)^{-1}=(id-(-(\delta^*\tilde\nabla-id)))^{-1}=\sum_{i=0}^N(-1)^i(\delta^*\tilde\nabla-id)^i.$$
\\Now we define a linear differential operator $L:\Gamma(V_0)\rightarrow\Gamma(V)$ by $$L(s)=\Sigma-(\delta^*\tilde\nabla)^{-1}\delta^*\tilde\nabla\Sigma,$$ where $\Sigma$ is a section of $V$ with $p_0(\Sigma)=s$ and $p_0:V\rightarrow V_0$ the projection. This is well defined, since $\Sigma$ is determined up to adding sections of $\textrm\im(\delta^*)$ and $L$ is zero on $\im(\delta^*)$. The operator $L$ obviously splits the projection $p_0$, i.e. $p_0(L(s))=s$. In addition, since $\delta^*\tilde\nabla(\delta^*\tilde\nabla)^{-1}$ is the identity on $\Gamma(\im(\delta^*))$, we see that $\delta^*\tilde\nabla L=0$. The operator $L$ is uniquely characterised by these two properties, since for a section $\Sigma\in\Gamma(V)$ with $p_0(\Sigma)=s$ and
$\delta^*\tilde\nabla\Sigma=0$, we obtain $L(s)=\Sigma-(\delta^*\tilde\nabla)^{-1}\delta^*\tilde\nabla\Sigma=\Sigma$. In particular, this shows that a section $\Sigma$ of $V$ lies in the image of $L$ if and only if $\delta^*\tilde\nabla\Sigma=0$.
\\Inserting the formula for $\delta^*\tilde\nabla$ and using that $\delta^*\partial$ is the identity on $\im(\delta^*)$, we obtain
$$L(s)=\sum_{i=0}^N(-1)^i(\delta^*\nabla)^i(\Sigma)-\sum_{i=0}^N(-1)^i(\delta^*\nabla)^i\delta^*\partial(\Sigma).$$
\\Since the formula of $L$ is independent of the choice of $\Sigma$, this implies that 
\begin{equation}\label{Lformula}
L(s)=\Sigma^N_{i=0}(-1)^i(\delta^*\nabla)^{i}(s), 
\end{equation}
where $s$ is viewed as a section of $V$ by trivial extension.
\\Denoting by $L^j$ the component in $V_0\oplus...\oplus V_j$ of $L^j$, we have:

\begin{prop}\label{Splittingop}
There exists a unique linear differential operator $L:\Gamma(V_0)\rightarrow \Gamma(V)$ such that
\begin{itemize}
\item{$p_0(L(s))=s$}
\item{$L$ has values in the kernel of $\delta^*\tilde\nabla$}
\end{itemize}
In particular, a section $\Sigma\in\Gamma(V)$ is in $\im(L)$ if and only if $\delta^*\tilde\nabla(\Sigma)=0$.
Moreover, each operator $L^j:\Gamma(V_0)\rightarrow \Gamma(V_0\oplus...\oplus V_j)$ induces a surjective vector bundle map 
$$\J^j(V_0)\rightarrow V_0\oplus...\oplus V_j,$$ which is an isomorphism for $j<r$.
\end{prop}

\begin{proof}
It only remains to show the last assertion. 
The principal connection on $\G_0^{ss}$ induces not only a linear connection $\nabla$ on $V$, but also a linear connection $\nabla$ on $\gr(TM)\cong TM$ satisfying $\nabla:\Gamma(\gr_i(TM))\rightarrow\Gamma(\gr(TM)^*\otimes\gr_i(TM))$. The $G_0^{ss}$-equivariance of $\delta^*: \g_-^*\otimes\V\rightarrow\V$ implies that the corresponding vector bundle map is parallel for the induced linear connection on $\gr(TM)\otimes V^*\otimes V$.
Therefore we conclude that $L(s)=\Sigma^N_{i=0}(-1)^i(\delta^*\nabla)^{i}s$ can be written as $$L(s)=\sum_{i=0}^N(-1)^i (\delta^*\circ(id\otimes\delta^*)\circ...\circ(\underbrace{id\otimes...\otimes id}_{i-1}\otimes \delta^*))\circ \nabla^i s$$ with the convention that the $0$-th term is the identity.
\\Denote by  $\mathcal T_{-i}(\gr(TM))=\G_0^{ss}\times_{G_0^{ss}}\mathcal T_{-i}(\g_-)$ the associated vector bundle corresponding to the $-i$-th grading component of the tensor algebra
$\T(\g_-)$ and consider the following differential operator
$$D^j: \Gamma(V_0)\rightarrow \Gamma(\bigoplus_{i=0}^j\mathcal T_{-i}(\gr(TM))^*\otimes V_0)$$
$$s\mapsto(\sum_{i=0}^j \nabla^is)_{\leq j}$$
where $(\quad)_{\leq j}$ means that we restrict $\sum_{i=0}^j \nabla^is$ to all grading components of degree $\leq j$ in $\bigoplus_{i=0}^j(\gr(TM)^i)^*\otimes V_0$.
This operator is obviously of weighted order $j$. Note that we have
$$\nabla\nabla s(\xi,\eta)-\nabla\nabla s(\eta,\xi)=R(\xi,\eta)(s)+\nabla_{\nabla_{\eta}\xi}s-\nabla_{\nabla_{\xi}\eta}s-\nabla_{[\eta,\xi]s}$$
and so
$$\nabla\nabla s(\xi,\eta)-\nabla\nabla s(\eta,\xi)-\nabla_{\mathcal L(\xi,\eta)}s\equiv 0 \textrm{ mod}( \textrm{ terms of lower weighted order in } s).$$
Therefore we conclude that the weighted symbol of $D^j$ is given by the canonical inclusion
$$\sigma(D^j): \U_{-j}(\gr(TM))^*\otimes V_0\hookrightarrow \mathcal T_{-j}(\gr(TM))^*\otimes V_0\subset\bigoplus_{i=0}^j\mathcal T_{-i}(\gr(TM))^*\otimes V_0,$$ which is obtained by dualising the projection $\mathcal T_{-j}(\gr(TM))\rightarrow\U_{-j}(\gr(TM)).$
Since $\delta^*$ is grading preserving, we deduce that
$$L^j(s)=\sum_{i=0}^j (-1)^i ((\delta^*\circ(id\otimes\delta^*)\circ...\circ(id\otimes...\otimes id\otimes \delta^*))\circ \nabla^i s)_{\leq j}=$$
$$=(\sum_{i=0}^j (-1)^i \delta^*\circ(id\otimes\delta^*)\circ...\circ(id\otimes...\otimes id\otimes \delta^*))\circ D^j(s)$$ is of weighted order $j$ and
hence induces a vector bundle map
$$L^j: \mathcal J^j(V_0)\rightarrow V_0\oplus...\oplus V_j.$$
Since $L^0$ is just the identity on $V_0$, the assertion holds for $j=0$.
\\Suppose now that $j\geq 1$ and let us compute the weighted symbol $\sigma(L^j)$ of $L^j$. It is given by the composition of the weighted symbol of $D^j$ with $\psi_j:=\sum_{i=1}^j(-1)^i((\delta^*\circ(id\otimes\delta^*)\circ...\circ(\underbrace{id\otimes...\otimes id}_{i-1}\otimes \delta^*))$
\begin{equation*}
\xymatrix{\U_{-j}(\gr(TM))^*\otimes V_0 \ar @{^{(}->}[r]^{\quad\sigma(D^j)}\ar[dr]_{\sigma(L^j)} & {\T_{-j}(\gr(TM))^*\otimes V_0\ar[d]^{\psi_j}}\\  & V_j\\}
\end{equation*}
Now consider the injective vector bundle map corresponding to the $G_0^{ss}$-equivariant inclusion of Proposition \ref{U_iV_0} $$\phi_j: V_j\rightarrow\U_{-j}(\gr(TM))^*\otimes V_0\subset \T_{-j}(\gr(TM))^*\otimes V_0.$$ This vector bundle map can also be written as
$$\phi_j=\sum_{i=1}^j(-1)^{i-1}p_0^j\circ (\underbrace{id\otimes...\otimes id}_{i-1}\otimes\partial)\circ...\circ(id\otimes\partial)\circ\partial, $$
where $p_0^j: \bigoplus_{i=1}^j\T_{-i}(\gr(TM))^*\otimes V_{j-i}\rightarrow\T_{-j}(\gr(TM))^*\otimes V_{0}$ is the projection given by restriction.
\\Setting $\partial^{(i)}:=\underbrace{id\otimes...\otimes id}_{i-1}\otimes\partial\circ...\circ(id\otimes\partial)\circ\partial|_{V_j}$ and $\delta^*_{(i)}:=\underbrace{id\otimes...\otimes id}_{i-1}\otimes\delta^*$,
we obtain that
\begin{align}\label{Last}
&(\sum_{i=1}^j(-1)^i(\delta^*\circ...\circ(\underbrace{id\otimes...\otimes id}_{i-1}\otimes \delta^*))\circ(\sum_{i=1}^j(-1)^{i-1} p^j_0\circ(\underbrace{id\otimes...\otimes id}_{i-1}\otimes\partial)\circ...\circ\partial)\nonumber\\
&=-[...\delta^*_{(j-2)}(p_0^j\circ\partial^{(j-2)}+\delta_{(j-1)}^*(p_0^j\circ\partial^{(j-1)}+\delta_{(j)}^*\circ p_0^j\circ\partial^{(j)}))].
\end{align}
Recall that $\delta^*:\gr(TM)^*\otimes V\rightarrow V$ is defined as the inverse of $\partial$ on $\im(\partial)\subset \gr(TM)^*\otimes V$ and zero on the rest.
Since $p_0^j\circ\partial^{(j)}=\partial^{(j)}$, we therefore get that $$\delta_{(j)}^*\circ p_0^j\circ\partial^{(j)}=\delta_{(j)}^*\circ\partial^{(j)}=(id-p^j_0)\circ\partial^{(j-1)}.$$ Hence $p_0^j\circ\partial^{(j-1)}+\delta_{(j)}^*\circ p_0^j\circ\partial^{(j)}=\partial^{(j-1)}$. Since $\delta_{(j-1)}^*\circ\partial^{(j-1)}$ equals again $(id-p^j_0)\circ\partial^{(j-2)}$, we conclude inductively that the composition $\sigma(L^j)\circ\phi_j$, which coincides with (\ref{Last}), equals $-id$ on $V_j$. In particular, $\sigma(L^j)$ is surjective and so is $L^j$.
Since we know by Proposition \ref{U_iV_0}  that $\phi_j: V_j\rightarrow\U_{-j}(\gr(TM))^*\otimes V_0$ is an isomorphism for $j<r$, we conclude that for $j<r$ the weighted symbol $\sigma(L^j)$ is an isomorphism, which equals $-\phi_{j}^{-1}$.
Therefore it follows by induction from the commutative diagram

\begin{equation}\label{DiaSplit}
\xymatrix{0\ar[d]&0\ar[d]\\ \mathcal U_{-j}(\gr(TM))^*\otimes V_0\ar[d]^{\iota}\ar[r]^{\quad\quad\quad\sigma(L^j)}& V_j \ar[d]\ar[r] & 0\\ \J^j(V_0)\ar[r]^{L^j}\ar[d]^{\pi^j_{j-1}}&V_0\oplus...\oplus V_j\ar[d]\ar[r]& 0\\\J^{j-1}(V_0)\ar[r]^{L^{j-1}}\ar[d]&V_0\oplus...\oplus V_{j-1}\ar[d]\ar[r] & 0\\0&0\\}
\end{equation}
that $L^j$ induces an isomorphism $\J^j(V_0)\rightarrow V_0\oplus...\oplus V_j$ for $j<r$
\end{proof}

\subsection*{3. Step - The construction of the bundle map $C$}

Now we define the following linear differential operator $$D^{\nabla}:=-(id\otimes\phi_{r-1})\circ\pi\circ\tilde\nabla\circ L:\Gamma(E)\rightarrow \Gamma(\circledcirc ^r \gr_{-1}(TM)^*\circledcirc E),$$
where $\pi$ denotes the projection $$\pi: \gr(TM)^*\otimes V\rightarrow\gr_{-1}(TM)^*\otimes V_{r-1}\rightarrow \ker(\square).$$
Since the projection $\pi$ annihilates $\im(\partial)$, we obtain that $$D^{\nabla}(s)=-(id\otimes\phi_{r-1})\pi\nabla(Ls)_{r-1},$$
where $(Ls)_{r-1}$ denotes the component in $V_{r-1}$ of $L(s)$.
From Proposition \ref{Splittingop} we know that $s\mapsto L(s)_{r-1}$ is a differential operator of weighted order $r-1$ with weighted symbol given by $-\phi_{r-1}^{-1}$ and so we see that $D^{\nabla}$ is of weighted order $r$ with weighted symbol given by
$$\sigma(D^{\nabla})=-id\otimes\phi_{r-1}\circ\pi\circ(-id\otimes\phi_{r-1}^{-1})$$
$$(\gr_{-1}(TM)^*\otimes\U_{r-1}(\gr(TM))^*\otimes E)\cap (\U_{-r}(\gr(TM))^*\otimes E)\rightarrow  \circledcirc ^r \gr_{-1}(TM)^*\circledcirc E.$$
\\Using (\ref{Isocohom}) we conclude that $\sigma(D^{\nabla})$ equals the projection (\ref{naturalprojection}).
Similarly as it was done for overdetermined systems on regular infinitesimal flag structures corresponding to $|1|$-graded semisimple Lie algebras 
in \cite{BCEG}, we can now start to rewrite the equation $D(s)=0$ for a differential operator $D$ as in Theorem \ref{main1}.

\begin{prop}\label{propA}
For any semi-linear differential operator $$D: \Gamma(E)\rightarrow\Gamma(\circledcirc^r \ggr_{-1}(TM)^*\circledcirc E)=\Gamma(F)$$
of weighted order $r$ with weighted symbol given by the projection (\ref{naturalprojection}), there exists a bundle map $A: V_0\oplus...\oplus V_{N}\rightarrow F$ such that $L$ and the projection $p_0: V\rightarrow V_0=E$ induce inverse bijections between the following sets
$$\{s\in\Gamma(E): Ds=0\}\leftrightarrow\{\Sigma\in\Gamma(V): \widetilde{\nabla}(\Sigma)+A(\Sigma)\in \Gamma(\textrm{\emph{im}}(\delta^*))\}.$$
\end{prop}

\begin{proof}
The operators $D$ and $D^{\nabla}$ have the same weighted symbol and therefore there exists a bundle map $\psi: \J^{r-1}(E)\rightarrow F$ such that $D(s)=D^\nabla(s)+\psi(j^{r-1}s).$ By Proposition \ref{Splittingop} the splitting operator $L$ induces an isomorphism $$L^{r-1}: \J^{r-1}(E)\cong V_0\oplus...\oplus V_{r-1}$$ and so there is a unique bundle map
$$A: V _0\oplus...\oplus V_{r-1}\rightarrow F \textrm{ \quad such that\quad } \psi(j^{r-1}s)=A(Ls)$$ and we view $A$ as a map on the whole bundle $V$ by trivial extension.
\\Since $\tilde\nabla Ls$ has values in $\ker(\delta^*)$ by Proposition \ref{Splittingop} and $A(Ls)$ even in $\ker(\square)\subseteq\ker(\delta^*)$, we obtain that $$0=D(s)=\pi(\tilde{\nabla}Ls+A(Ls))$$ if and only if
$$\tilde\nabla Ls+A(Ls)\in\Gamma(\textrm{im}(\delta^*)),$$ where $\pi:T^*M\otimes V\rightarrow\ker(\square)$ is the projection.
\\Conversely, suppose $\Sigma$ is a section of $V$ such that $\tilde\nabla\Sigma+A(\Sigma)\in\Gamma(\textrm{im}(\delta^*))$. Then $\delta^*(\tilde\nabla\Sigma+A(\Sigma))=0$ and since the map $A$ has values in $\ker(\delta^*)$, we get $\delta^*(\tilde\nabla\Sigma)=0$. By Proposition \ref{Splittingop} the equality $\delta^*(\tilde\nabla\Sigma)=0$ implies that $\Sigma=L(p_0(\Sigma))$ and hence $D(p_0(\Sigma))=0$.
\end{proof}
The fact that $A: V\rightarrow\ker(\square)\subset T^*M\otimes V$ is of homogeneity $\geq 1$, allows us to compute the section $\tilde\nabla\Sigma+A(\Sigma)\in\Gamma(\im(\delta^*))$.

\begin{prop}\label{propB1}
Let $A:V\rightarrow T^*M\otimes V$ be a bundle map of homogeneity $\geq 1$. Then there exists a differential operator $B:\Gamma(V)\rightarrow \Gamma(T^*M\otimes V)$ such that
$$\tilde\nabla\Sigma+A(\Sigma)\in\Gamma(\textrm{\emph{im}}(\delta^*))\quad\textrm{ if and only if }\quad\quad\tilde\nabla\Sigma+B(\Sigma)=0.$$
If $A$ is a vector bundle map, $B$ is a linear differential operator.
\end{prop}

\begin{proof}
Since we have chosen a splitting of the tangent bundle, we can identify $TM$ with $\gr(TM)$. Therefore we have a grading on differential forms with values in $V$ corresponding to the grading (\ref{gradLambda}) on $\Lambda^n\g_-^*\otimes \V$, which is given by homogeneous degree. We denote by lower indices the grading components
$$(\Lambda^nT^*M\otimes V)_\ell=(\Lambda^n\gr(TM)^*\otimes V)_\ell:=\bigoplus_{j=n}^{nk}(\Lambda^n_{-j}\gr(TM))^*\otimes V_{\ell-j}$$ and by upper indices the filtration components $(\Lambda^n T^*M\otimes V)^\ell$ of the associated filtration ($\phi\in(\Lambda^nT^*M\otimes V)^\ell\textrm{ if and only if }\phi(T^{i_1}M,...,T^{i_n}M)\subset V^{i_1+....+i_n+\ell}$).
Further, let us denote by $d^{\tilde\nabla}$ the covariant exterior derivative corresponding to the linear connection $\tilde\nabla$.
Recall that for a one form $\phi\in\Gamma(T^*M\otimes V)$ the covariant exterior derivative is given by
\begin{equation}\label{d}
d^{\tilde\nabla}\phi(\xi,\eta)=\tilde\nabla_\xi(\phi(\eta))-\tilde\nabla_\eta(\phi(\xi))-\phi([\xi,\eta]).
\end{equation}
Inserting $\phi=\tilde\nabla\Sigma$ into (\ref{d}) we see that $d^{\tilde\nabla}\tilde\nabla\Sigma(\xi,\eta)$ equals the curvature $\tilde R(\xi,\eta)(\Sigma)$ of $\tilde\nabla$. We will write $\tilde R(\Sigma)$ for the two form, which is given by $(\xi,\eta)\mapsto\tilde R(\xi,\eta)(\Sigma)$.
\\Now let us consider the equation $\tilde\nabla\Sigma+A(\Sigma)=\delta^*\psi$ for some $\psi\in\Lambda^2T^*M\otimes V$ and show how it can be rewritten. Concerning the bundle map $A$, we will write $A_i(\Sigma)$ for the $i$-th grading component and  $A^i(\Sigma)$ for the $i$-th filtration component of $A(\Sigma)$ in $T^*M\otimes V$.
Since $\delta^*$ is filtration preserving, we have $\delta^*\psi\in(T^*M\otimes V)^2\subset(T^*M\otimes V)^1=T^*M\otimes V$ and we set $B_1(\Sigma):=A_1(\Sigma)$.
Then the equation reads as
\begin{equation}\label{B}
\tilde\nabla\Sigma+B_1(\Sigma)+A^2(\Sigma)=\delta^*\psi.
\end{equation}
Since $\tilde\nabla$ is of homogeneity $\geq 0$ and its lowest homogeneous component is given by $\partial$, the same is true for $d^{\tilde\nabla}: T^*M\otimes V\rightarrow \Lambda^2T^*M\otimes V$. Hence the operator
$\delta^*d^{\tilde\nabla}: T^*M\otimes V\rightarrow T^*M\otimes V$ is also of homogeneity $\geq 0$ with lowest homogeneous component $\delta^*\partial$, which by definition of $\delta^*$ is the identity on $\im(\delta^*)\subset T^*M\otimes V$. Applying
$\delta^*d^{\tilde\nabla}$ to the equation (\ref{B}), we can therefore compute the lowest grading component $(\delta^*\psi)_2$. Moving the resulting expression for $(\delta^*\psi)_2$ to the other side of the equation and applying $\delta^*d^{\tilde\nabla}$ to the new equation, we can compute $(\delta^*\psi)_3$ and so on until we have computed the whole one form $\delta^*\psi$.
More explicitly, if we apply first $d^{\tilde\nabla}$ to the equation (\ref{B}), we obtain that $$\tilde R(\Sigma)+d^{\tilde\nabla} B_1(\Sigma)+d^{\tilde\nabla}A^2(\Sigma)=d^{\tilde\nabla}\delta^*\psi.$$ This implies the following equation for the second grading component $\partial ((\delta^*\psi)_2)$ of $(d^{\tilde\nabla}\delta^*\psi)$ $$(\tilde R(\Sigma)+d^{\tilde\nabla}B_1(\Sigma))_2+\partial (A_2(\Sigma))=\partial ((\delta^*\psi)_2).$$ Applying now $\delta^*$ we see that $$\delta^*((\tilde R(\Sigma)+d^{\tilde\nabla}B_1(\Sigma))_2+\partial A_2(\Sigma))=\delta^*\partial ((\delta^*\psi)_2)=(\delta^*\psi)_2,$$ since $\delta^*\partial$ is the identity on $\im(\delta^*)$.
\\If we set $B_2(\Sigma):=A_2(\Sigma)-\delta^*((\tilde R(\Sigma)+d^{\tilde\nabla}B_1(\Sigma))_2+\partial A_2(\Sigma))$, the equation (\ref{B}) can be now written as $$\tilde\nabla(\Sigma)+B_1(\Sigma)+B_2(\Sigma)+A^3(\Sigma)=(\delta^*\psi)^3.$$ Applying again $\delta^*d^{\tilde\nabla}$, we can compute $(\delta^*\psi)_3$ and define $B_3$ by substracting the resulting expression for $(\delta^*\psi)_3$ from $A_3(\Sigma)$.
In this way, we can inductively define $$B_i(\Sigma):=A_i(\Sigma)-\delta^*([\tilde R(\Sigma)+d^{\tilde\nabla}(B_1(\Sigma)+...+B_{i-1}(\Sigma))]_{i}+\partial(A_i(\Sigma)).$$ Defining the differential operator $B$ as $B(\Sigma):=\sum_{i=1}^{N+k}B_i(\Sigma)$, it has by construction the property required in the proposition and we are done.
\end{proof}

To compute the weighted order of the differential operator $B$, we need a bit of information about the curvature $\tilde R$ of $\tilde\nabla$.

 \begin{lem}\label{curvature}
 Let $R\in\Lambda^2T^*M\otimes V$ be the curvature of the connection $\nabla$ on $V$ and $T$ the torsion of the connection $\nabla$ on $TM\cong\ggr(TM)$. Then the curvature of $\tilde\nabla=\nabla+\partial$ is given by
 $$\tilde R(\xi,\eta)(\Sigma)=R(\xi,\eta)(\Sigma)+\partial(\Sigma)(T(\xi,\eta)+\mathcal L(\xi,\eta)).$$
 Moreover, the map $\Sigma\mapsto \tilde R(\Sigma)$ is of homogeneity $\geq 1$.
 \end{lem}
 
 \begin{proof}
 The curvature of $\tilde\nabla$ is given by $\tilde R(\xi,\eta)(\Sigma)=\tilde\nabla_\xi\tilde\nabla_\eta\Sigma-\tilde\nabla_\eta\tilde\nabla_\xi\Sigma-\tilde\nabla_{[\xi,\eta]}\Sigma.$
 For the first term we have
 \begin{equation}\label{R}
 \tilde\nabla_\xi\tilde\nabla_\eta\Sigma=\nabla_\xi\nabla_\eta\Sigma+\nabla_\xi(\partial(\Sigma)(\eta))+\partial(\nabla_\eta\Sigma)(\xi)+\partial(\partial\Sigma(\eta))(\xi).
 \end{equation}
The second summand of (\ref{R}) can be written as $\nabla_\xi(\partial(\Sigma)(\eta))=(\nabla_\xi(\partial\Sigma))(\eta)+\partial\Sigma(\nabla_\xi\eta).$ Since $\partial: \V\rightarrow\g_-^*\otimes\V$ is $G_0^{ss}$-equivariant, the induced vector bundle map is parallel and so we have $(\nabla_\xi(\partial\Sigma))(\eta)=\partial(\nabla_\xi\Sigma)(\eta)$. Putting this together, we obtain that
 $$\tilde\nabla_\xi\tilde\nabla_\eta\Sigma=\nabla_\xi\nabla_\eta\Sigma+\partial(\nabla_\xi\Sigma)(\eta)+\partial\Sigma(\nabla_\xi\eta)+\partial(\nabla_\eta\Sigma)(\xi)+\partial(\partial\Sigma(\eta))(\xi).$$ Therefore we have
 \begin{align}\label{R1}
 \tilde R(\xi,\eta)(\Sigma) &=R(\xi,\eta)(\Sigma)+\partial\Sigma(T(\xi,\eta))+\partial(\partial\Sigma(\eta))(\xi)-\partial(\partial\Sigma(\xi))(\eta) \nonumber\\
  &=R(\xi,\eta)(\Sigma)+\partial(\Sigma)(T(\xi,\eta)+\mathcal L(\xi,\eta)),
  \end{align}
 since $$0=\partial(\partial\Sigma)(\xi,\eta)=\partial(\partial\Sigma(\eta))(\xi)-\partial(\partial\Sigma(\xi))(\eta)-\partial\Sigma(\mathcal L(\xi,\eta)).$$
 \\Since $\tilde R(\Sigma)$ equals $d^{\tilde\nabla}\tilde\nabla(\Sigma)$, the map $\Sigma\mapsto\tilde R(\Sigma)$ is at least of homogeneity $\geq 0$. To see that is actually of homogeneity $\geq 1$ we consider the formula (\ref{R1}).
The curvature $\Sigma\mapsto R(\Sigma)=d^\nabla\nabla(\Sigma)$ of $\nabla$ is of homogeneity $\geq 2$, since $\nabla$ and $d^\nabla$ both are of homogeneity $\geq 1$.
Now consider the second term of (\ref{R1}) given by $\partial(\Sigma)(T(\xi,\eta)+\mathcal L(\xi,\eta))$. We have $$T(\xi,\eta)+\{\xi,\eta\}=\nabla_\xi\eta-\nabla_\eta\xi-[\xi,\eta]+\mathcal L(\xi,\eta)$$ and under the identification of $\gr(TM)$ with $TM$ we can view $\mathcal L(\xi,\eta)$ as the grading component of lowest degree $-(\textrm{ord}(\xi)+\textrm{ord}(\eta))$ of $[\xi,\eta]$. Therefore the two form
 $T+\mathcal L(\, ,\, )$ is of homogeneity $\geq 1$. This implies that $$\Sigma\mapsto\partial(\Sigma)(T+\mathcal L( \,,\, ))$$ is of homogeneity $\geq 1$, since $\partial$ is filtration preserving.
 \end{proof}

Using Lemma \ref{curvature} we are able to determine the weighted order of $B$:

\begin{prop}\label{propB2}
The differential operator $B$ is of weighted order $N+k-1$, where $k$ is the depth of the filtration of $TM$. Therefore it defines a bundle map
$$B: \J^{N+k-1}(V)\rightarrow T^*M\otimes V.$$
Moreover, the component $B_i$ factors through $\J^{i-1}(V_0)\oplus...\oplus \J^{1}(V_{i-2})\oplus V_{i-1}.$
\end{prop}

\begin{proof}
Let us write a section $\Sigma\in\Gamma(V)$ as $\Sigma=(\Sigma_0,...,\Sigma_N)$.
We shall prove the proposition by induction on $i$.
\\Since $A:V\rightarrow T^*M\otimes V$ is of homogeneity $\geq 1$, $B_1(\Sigma)=A_1(\Sigma)$ just depends on $\Sigma_0$ and so the assertion holds for $i=1$.
\\Now consider $B_2(\Sigma)=A_2(\Sigma)-\delta^*((\tilde R(\Sigma)+d^{\tilde\nabla}B_1(\Sigma))_2+\partial A_2(\Sigma))$.
The component $A_2(\Sigma)$ depends on $\Sigma_0$ and $\Sigma_1$, since $A$ is of homogeneity $\geq 1$. By Lemma \ref{curvature} we know that $\Sigma\mapsto \tilde R(\Sigma)$ is also of homogeneity $\geq 1$ and therefore $\tilde R(\Sigma)_2$ only depends on $\Sigma_0$ and $\Sigma_1$. So it remains to look at the term
$$(d^{\tilde\nabla}B_1(\Sigma))_2\in\Gamma(\gr_{-1}(TM)^*\wedge\gr_{-1}(TM)^*\otimes V_0).$$
Since $\partial$ is grading preserving, we have for $\xi,\eta\in \Gamma(T^{-1}M)$ that
\begin{align}
&d^{\tilde\nabla}B_1(\Sigma)(\xi,\eta)=\nonumber\\
&=\nabla_\xi(B_1(\Sigma)(\eta))-\nabla_\eta(B_1(\Sigma)(\xi))+\partial(B_1(\Sigma)(\eta))(\xi)-\partial(B_1(\Sigma)(\xi))(\eta)-B_1(\Sigma)([\xi,\eta])\nonumber\\
&=\nabla_\xi(B_1(\Sigma)(\eta))-\nabla_\eta(B_1(\Sigma)(\xi))+\partial (B_1(\Sigma))(\xi,\eta)-B_1(\Sigma)([\xi,\eta]-\mathcal L(\xi,\eta))\nonumber\\
&=\nabla_\xi(B_1(\Sigma)(\eta))-\nabla_\eta(B_1(\Sigma)(\xi))-B_1(\Sigma)([\xi,\eta]-\mathcal L(\xi,\eta))\nonumber,
\end{align}
Since the component $B_1(\Sigma)\in\gr_{-1}(TM)^*\otimes V_0$ just depends on $\Sigma_0$, we therefore conclude that $(d^{\tilde\nabla}B_1(\Sigma))_2$ depends on the weighted one jet of $\Sigma_0$. In total, we see that $B_2$ induces a bundle map $\J^1(V_0)\oplus V_1\rightarrow (T^*M\otimes V)_2$.
\\Now assume the statement is true for $B_i$ with $i<N+k$. The $i+1$-th component is given by
$$B_{i+1}(\Sigma):=A_{i+1}(\Sigma)-\delta^*([\tilde R(\Sigma)+d^{\tilde\nabla}(B_1(\Sigma)+...+B_{i}(\Sigma))]_{i+1}+\partial(A_{i+1}(\Sigma)).$$ Again, since $A$ and $\tilde R$ are of homogeneity $\geq 1$,
$A_{i+1}(\Sigma)$ and $(\tilde R(\Sigma))_{i+1}$, depend only on $\Sigma_0,....,\Sigma_i$. So it remains to study the term $$(d^{\tilde\nabla}(B_1(\Sigma)+...+B_{i}(\Sigma)))_{i+1}.$$
\\For $j<i+1$ consider $B_j(\Sigma)\in\Gamma(\bigoplus_{\ell=1}^j\gr_{-\ell}(TM)^*\otimes V_{j-\ell}).$
We know that the operator $d^{\tilde\nabla}$ is of homgeneity $\geq 0$ and hence we have $d^{\tilde\nabla}(B_j(\Sigma))\in\Gamma((\Lambda^2\gr(TM)^*\otimes V)^j)$.
Consider
\begin{align}
&(d^{\tilde\nabla}B_j(\Sigma))_{i+1}(\xi,\eta)=\nonumber\\
&=[\nabla_\xi(B_j(\Sigma)(\eta))-\nabla_\eta(B_j(\Sigma)(\xi))+\partial (B_j(\Sigma))(\xi,\eta)-B_j(\Sigma)([\xi,\eta]-\mathcal L(\xi,\eta))]_{i+1}\nonumber.
\end{align}
Obviously $d^{\tilde\nabla}(B_j(\Sigma))_{i+1}$ depends only on $B_j(\Sigma)$ and derivatives of $B_j(\Sigma)$ in direction of vector fields of order $i+1-j$. The claim now follows from the assumption that $B_j(\Sigma)$ factors through $\J^{j-1}(V_0)\oplus...\oplus V_{j-1}$ for all $j<i+1$.
\end{proof}

Now one can do the last step in rewriting the equation $D(s)=0$ by solving $\tilde\nabla\Sigma+B(\Sigma)=0$ component by component.

\begin{prop}\label{propC}
Suppose that  $B:\J^{N+k-1}(V)\rightarrow T^*M\otimes V$ is a bundle map such that its $i$-th component
$B_i: \J^{N+k-1}(V)\rightarrow(T^*M\otimes V)_i$ factors through
$\J^{i-1}(V_0)\oplus...\oplus \J^{1}(V_{i-2})\oplus V_{i-1}.$ Then there exists a bundle map $C:V\rightarrow T^*M\otimes V$ such that 
$$\tilde\nabla\Sigma+B(\Sigma)=0\quad\textrm{ if and only if }\quad \tilde\nabla\Sigma+C(\Sigma)=0.$$ If $B$ is a vector bundle homomorphism, then also $C$ can be chosen to be a vector bundle homomorphism.
\end{prop}

\begin{proof}
The linear connection $\tilde\nabla=\nabla+\partial$ is of homogeneity $\geq 0$ with lowest homogeneous component given by the vector bundle map $\partial$.
Since we have a linear connection $\nabla$ on $TM\cong\gr(TM)$, we can form iterated covariant derivatives $\tilde\nabla^i$. We know that the linear connection on $TM$ is of homogeneity $\geq 1$, since $\nabla: \Gamma(\gr_i(TM))\rightarrow\Gamma(\gr(TM)^*\otimes \gr_i(TM))$, and hence, since $\tilde\nabla$ is of homogeneity $\geq 0$ with lowest homogeneous component $\partial$, we conclude that the iterated covariant derivative $\tilde\nabla^i$ is also of homogeneity $\geq 0$ and that its lowest homogeneous component is algebraic. By assumption on $B$ we therefore deduce that the component $B_i$ just depends on $\Sigma_{\leq i-1}$, $(\tilde\nabla\Sigma)_{\leq i-1}$,...,$(\tilde\nabla^{i-1}\Sigma)_{\leq i-1}$ and we may write $$B_i(\Sigma)=B_i(\Sigma_{\leq i-1},(\tilde\nabla\Sigma)_{\leq i-1},...,(\tilde\nabla^{i-1}\Sigma)_{\leq i-1}),$$
where $(_-)_{\leq i-1}$ means that we restrict to grading components of degree $\leq i-1$.  Let us now consider the equation $\tilde\nabla(\Sigma)+B(\Sigma)=0$ grading component by grading component.
For the first component we get $(\tilde\nabla\Sigma)_1+B_1(\Sigma)=0$ and we set $C_1(\Sigma):=B_1(\Sigma_0)$. For the second component we have $$(\tilde\nabla\Sigma)_2+B_2(\Sigma_0,\Sigma_1,(\tilde\nabla\Sigma)_1)=0$$ and we define $C_2(\Sigma_0,\Sigma_1):=B_2(\Sigma_0,\Sigma_1,-C_1(\Sigma_0))$. By construction we have
\begin{equation}\label{C2}
((\tilde\nabla\Sigma)+B(\Sigma))_{\leq 2}=0 \quad\textrm{ if and only if }\quad((\tilde\nabla\Sigma)+C(\Sigma))_{\leq 2}=0,
\end{equation}
where $C=C_1+C_2$.
\\Suppose now inductively that we have found bundle maps $C_1,...,C_i$ for $i<N+k$ such that
\begin{equation}\label{Ci}
(\tilde\nabla\Sigma+B(\Sigma))_{\leq i}=0 \quad\textrm{ if and only if }\quad(\tilde\nabla\Sigma+C(\Sigma))_{\leq i}=0,
\end{equation}
where $C=C_1+...+C_i$ and $C_j$ depends only on $\Sigma_{\leq j-1}$.
Assume further that for any section $\Sigma$ satisfying (\ref{Ci})
we have derived algebraic expressions in terms of
$\Sigma_0,...,\Sigma_{\leq i-1}$ for all $(\tilde\nabla^\ell\Sigma)_{\leq i}$ with $\ell=1,...,i$.
Inserting these expressions into $B_{i+1}(\Sigma)$, we obtain a bundle map $C_{i+1}(\Sigma_0,...,\Sigma_i)$ such that
\begin{equation}\label{Ci+1}
(\tilde\nabla\Sigma+B(\Sigma))_{\leq i+1}=0\quad\textrm{ if and only if } \quad(\tilde\nabla\Sigma+C(\Sigma))_{\leq i+1}=0,
\end{equation}
$C=C_1+...+C_{i+1}$.
\\It remains to show that for any section $\Sigma$ satisfying (\ref{Ci+1}) we can deduce algebraic expressions in terms of $\Sigma_0,...,\Sigma_{i}$ for all $(\tilde\nabla^{\ell}\Sigma)_{\leq i+1}$ occurring in $B_{i+2}$, where $\ell=1,...,i+1$. Since $\tilde\nabla ^j$ is of homogeneity $\geq 0$,
$(\tilde\nabla\Sigma+C(\Sigma))_{\leq i+1}=0$ implies that $(\tilde\nabla^{j}(\tilde\nabla\Sigma+C(\Sigma)))_{\leq i+1}=0$.
The differential operator $$((\tilde\nabla^1 C(\Sigma))_{i+1},...,(\tilde\nabla^i C(\Sigma))_{i+1})$$ depends on the weighted $i$-jet of $C_1(\Sigma)$, on the weighted $i-1$-jet of $C_2(\Sigma)$,...,on the weighted one jet of $C_i(\Sigma)$ and algebraic on $C_{i+1}(\Sigma)$.  Therefore it just depends on $\Sigma_{\leq i}$, $(\tilde\nabla\Sigma)_{\leq i}$..., $(\tilde\nabla^i\Sigma)_{\leq i }$, for which we have by induction hypothesis algebraic formulae in terms of $\Sigma_0,...,\Sigma_i$. Hence we get formulae in terms of $\Sigma_0,..,\Sigma_i$ for $(\tilde\nabla^{j+1}\Sigma)_{\leq i+1}$ with $j=0,...,i$ and we are done. If $B$ is a linear differential operator, $C$ will be a vector bundle map by construction.
\end{proof}

Combining Propositions \ref{propA}, \ref{propB1}, \ref{propB2} and \ref{propC}, we have proved Theorem \ref{main1}. 
In particular, if a system $D(s)=0$ of the form of Theorem \ref{main1} is linear, the bundle map $C$ is a vector bundle map and solutions of $D(s)=0$ correspond bijectively to parallel sections of the linear connection $\tilde\nabla+C$. Since a parallel section is already determined by its value in a single point, we obtain as a corollary:

\begin{cor}\label{bound}
Let $\E$ be an irreducible representation of $G_0^{ss}$ and $r>0$ be an integer.
For a linear differential operator $D:\Gamma(E)\rightarrow\Gamma(\circledcirc^r\ggr_{-1}(TM)^*\circledcirc E)$ of weighted order $r$ with weighted symbol given by (\ref{naturalprojection}) the solution space of $D(s)=0$ is finite dimensional and bounded by the dimension of the irreducible $G$ representation $\V[\E,r]$.
Moreover, if the grading (\ref{decomV1}) of $\V[\E,r]$ is of the form $\V[\E,r]=\V_0\oplus...\oplus\V_N$ a solution is already determined by its weighted $N$-jet in a single point.
\end{cor}

\begin{rem}
In nearly all cases a regular infinitesimal flag structure on a manifold $M$ determines a regular normal parabolic geometry of the same type. A large class of invariant differential operators for parabolic geometries occur as differential operators in some BGG-sequence, see \cite{CSS} and \cite{CD}. If the center of $\g_0$ is one dimensional the first operator in a BGG-sequence is always a linear differential operator of the form described in Theorem \ref{main1} and hence the prolongation procedure presented here applies to them. On the one hand this shows that Theorem \ref{main1} covers a lot of geometrically interesting equations, like the equation for the infinitesimal automorphisms or in the case of conformal geometries the equations for conformal Killing tensors, the equation for twistor spinors and the equation for Einstein scales.
On the other hand it shows that the bound in Corollary \ref{bound} is sharp. In fact, considering the homogenous space $G/P$ with its canonical regular infinitesimal flag structure of type $(G,P)$ (see \cite{CSbook}), it turns out that in this case $D^{\nabla}$ equals the first operator in the BGG-sequence corresponding to $V$ and that solutions of $D^{\nabla}(s)=0$ correspond to parallel sections of the flat tractor connection on $V$, see \cite{CG}. The flatness of the connection implies that the dimension of the solution space equals the rank of $V$.
\end{rem}

\begin{rem}
For the first BGG-operators it was recently shown in \cite{HSSS} how to construct a linear connection on the corresponding tractor bundle $V$, whose parallel sections correspond bijectively to solutions of the linear system of equations defined by a first BGG-operator. This approach has the feature that the connection on $V$ is natural with respect to the parabolic geometry respectively its underlying regular infinitesimal flag structure. Our approach in contrast, although not invariant, works for a larger class of operators, namely all semi-linear operators having the same weighted symbol as some first BGG-operator. Note also that to apply Theorem \ref{main1} one just has to check if the operator in question has the right weighted symbol and one doesn't need to know, if one is dealing with a BGG-operator.   
\end{rem}

\subsection{Applications to overdetermined systems on contact manifolds}\label{AppCon}

For $n\geq 1$ consider $\R^{2n+2}$ endowed with the skew-symmetric non-degenerate bilinear form $$<(x_0,...,x_{2n+1}), (y_0,...,y_{2n+1})>=x_0y_{2n+1}-y_0x_{2n+1}+\sum_{i=1}^n(x_iy_{n+i}-x_{n+i}y_i).$$
Moreover, let $$\g=\s\p(2n+2,\R)=\{A\in \textrm{End}(\R^{2n+2}): <Ax,y>=-<x,Ay> \textrm{for all} x,y\in\R^{2n+2}\}$$ be the symplectic Lie algebra with respect to $< \,,\, >$.
\\It turns out that $\g$ is given by block matrices of block sizes $1$, $n$, $n$ and $1$ of the following form:
$$\g=\left\{\left(\begin{matrix} a & Z  & W & z\\X & A & B&W^t \\
Y&C & -A^t&-Z^t\\x& Y^t&-X^t&-a\\\end{matrix}\right): B^t=B, C^t=C\right\}.$$
This realisation of $\g$ defines a $|2|$-grading on $\g=\g_{-2}\oplus\g_{-1}\oplus\g_0\oplus\g_1\oplus\g_2$ given by
$$\left(\begin{matrix} \g_0 & \g_1 &\g_1&\g_2 \\\g_{-1} &\g_0&\g_0&\g_1\\\g_{-1}&\g_0&\g_0&\g_{1} \\
\g_{-2} &\g_{-1} & \g_{-1}&\g_0\\\end{matrix}\right),$$ where the subalgebra $\g_-$ is isomorphic to a Heisenberg Lie algebra. Hence we have a contact grading 
(see example \ref{excontact}). 
Note that $\g_0\cong\R\oplus\s\p(2n,\R)$, where $\s\p(2n,\R)$ is the symplectic Lie algebra with respect to the standard symplectic form on $\R^{2n}$.
Let $G=Sp(2n+2,\R)$ be the symplectic Lie group consisting of linear symplectic automorphisms of $(\R^{2n+2}, < , >)$ and let $P\subset G$ be the parabolic subgroup with Lie algebra $\p=\g_0\oplus \g_1\oplus\g_2$ given by the connected component of the identity of all block upper triangular matrices in $G$ with block sizes $1$, $n$, $n$ and $1$.
The corresponding Levi subgroup $G_0$ is given by all the block diagonal matrices in $P$

$$G_0=\left\{\left(\begin{matrix} e & & \\\ & D &  \\
& & e^{-1}\\\end{matrix}\right): D\in Sp(2n,\R), e\in\R_{>0} \right\},$$
where $Sp(2n,\R)$ is the symplectic Lie group wit respect to the standard symplectic form on $\R^{2n}$.
A regular infinitesimal flag structure of type $(G, P)$ on a manifold $M$ consists of a contact structure $TM=T^{-2}M\supset T^{-1}M=:H$ and a reduction $\G_0$ of the structure group of $\mathcal P(\gr(TM))$ via $Ad: G_0\rightarrow Aut_{gr}(\g_-)$. It is easy to see that $G_0$ can be identified via $Ad$ with the subgroup consisting of those automorphisms in $Aut_{gr}(\g_-)$, which in addition preserve an orientation on $\g_-$. Therefore a regular infinitesimal flag manifold $(M, H, \G_0)$ of type $(G,P)$ is just an oriented contact manifold.
\\Recall that for an orientable contact manifold $(M,H)$, there exists a globally defined \textit{contact form}, i.e. a section $\alpha\in\Gamma(T^*M)$ such that $\ker(\alpha)=H$. 
It is unique up to multiplication by a nowhere vanishing function and $\alpha\wedge(d\alpha)^n$ is a volume form on $M$. Note that the choice of a contact from reduces the frame bundle of $\gr(TM)$ further to $G_0^{ss}=Sp(2n,\R)$.
Moreover, it is well known that a contact form $\alpha$ gives rise to a unique vector field $r$  such that $\alpha(r)=1$ and $i_rd\alpha=d\alpha(r,_-)=0$. It is called the \textit{Reeb vector field} associated to $\alpha$. In particular, $\alpha$ induces a splitting of the filtration of the tangent bundle $TM\cong \gr(TM)$ given by $\xi\mapsto (\xi-\alpha(\xi)r,\alpha(\xi)).$
\\Now suppose that $(M,H, \G_0)$ is an oriented contact manifold. Further, assume that we have chosen a contact form $\alpha$ inducing the given orientation on $M$ and let $\G_0^{ss}$ be the corresponding reduction to $Sp(2n,\R)$. Moreover, we identify $TM$ and $\gr(TM)$ via the isomorphism induced from the associated Reeb vector field.
Hence Theorem \ref{main1} applies to geometric structures of the form $(M, H, \G_0^{ss})$ and Corollary \ref{bound} reads as follows:

\begin{cor}\label{contactcor}
Suppose that $\E$ is an irreducible representation of $Sp(2n,\R)$ and let $r>0$ be an integer.
Then for every linear differential operator $D:\Gamma(E)\rightarrow\Gamma(S^r H^*\circledcirc E)$ of weighted order $r$ with symbol given by the natural projection
$$\sigma(D):\mathcal U_{-r}(\ggr(TM))^*\otimes E\rightarrow S^r H^*\circledcirc E.$$
 the dimension of the solution space of the associated linear system $Ds=0$ is bounded by the dimension of the irreducible $Sp(2n+2,\R)$-representation $\V[\E,r]$.
\end{cor}

\begin{ex}
Choosing a principal connection on $\G_0^{ss}$, we obtain linear connections on all associated vector bundles. In particular, we get a partial linear connection $\nabla:\Gamma(H^*)\rightarrow\Gamma(H^*\otimes H^*)$ on $H^*$. Now consider the following linear differential operator $D:\Gamma(H^*)\rightarrow\Gamma(S^{r+1}H^*)$ given by
$$D: s_b\mapsto\nabla_{(a_1}\nabla_{a_2}.....\nabla_{a_r}s_{b)},$$ where one symmetrises over all the indices in the round bracket. It is a differential operator of weighted order $r$ with symbol given by the natural projection $\mathcal U_{-r}(\gr(TM))^*\otimes H^*\rightarrow S^r H^*\circledcirc H^*=S^{r+1}H^*$. By Corollary \ref{contactcor} the dimension of the solution space of $D(s)=0$ is bounded by the dimension of $\V[\g_{-1}^*,r]$.
Using elementary representation theory one immediately computes that
$$\dim(\V[\g_{-1}^*,r])=\frac{(2n+r)!(2n)!r(2n+2+r)}{(r+1)!(2n+1)!(2n-1)!}.$$
\end{ex}

\begin{ex}
Choosing a principal connection on $\G_0^{ss}$ induces partial connections on $H^*$ and $S^tH^*$, which we denote both by $\nabla$.
Now consider the linear differential operator $D:\Gamma(S^tH^*)\rightarrow\Gamma(S^{t+r}H^*)$ given by
$$s_{b_1...b_t}\mapsto \nabla_{(a_1}\nabla_{a_2}...\nabla_{a_r}s_{b_1....b_t)},$$ where one symmetrises over all the indices in the round bracket. It is of weighted order $r$ and the weighted symbol is given by the projection $\U_{-r}(\gr(TM))^*\otimes S^tH^*\rightarrow S^{r}H^*\circledcirc S^{t}H^*=S^{t+r}H^*$.
By Corollary \ref{contactcor} its solution space is bounded by the dimension of $\V[S^t\g_{-1}^*,r]$ and one directly computes that
$$\dim(\V[S^t\g_{-1}^*,r])=\frac{(r+t+2n-1)!(2n+t-1)!r(r+2t+2n)}{(r+t)!t!(2n+1)!(2n-1)!}.$$
\end{ex}

\subsection{Semi-linear systems on regular infinitesimal flag manifolds corresponding to $|k|$-gradings such that $\dim(\z(\g_0))>1$}\label{general}

Suppose that $\g=\g_{-k}\oplus ....\oplus\g_0\oplus...\oplus\g_k$ is a complex $|k|$-graded semisimple where the center $\z(\g_0)$ of the Levi subalgebra has dimension $d>1$. Let $G$ be a simply connected Lie group with Lie algebra $\g$ and $P\subset G$ be a parabolic subgroup corresponding to the grading on $\g$ with Levi subgroup $G_0$ . Further, assume that $M$ is a manifold endowed with a geometric structure $(\mathcal G_0^{ss},\{T^iM\})$ of type $(G,P)$ as described in Section \ref{RegInfFlag}.
\\From Lemma \ref{g_1} we know that $\g_{-1}$  decomposes as $\g_0^{ss}$-module into irreducible as follows
$$\g_{-1}=\bigoplus_{j\in J}\g_{-1,j}$$
where $\g_{-1,j}$ is the irreducible representation with highest weight $-\alpha_j\in\Sigma_\p$ and $J$ as in Lemma \ref{g_1}. Recall that the number of elements of $J$ is $d$. 
\\Suppose that $\E$ is an irreducible representation of $G_0^{ss}$ and for $j\in J$ fix an element $r_j\in\N$. If the restriction of $\lambda=\sum_{i\in I\backslash J}a_i\omega_{\alpha_i}$ to $\h_0$ is the highest weight of $\E^*$, we define $\V$ as the irreducible representation of $\g$, which is dual to the irreducible representation with highest weight $\lambda+\sum_{j\in J}(r_j-1)\omega_{\alpha_j}\in\h^*$. Using Theorem \ref{Kostant}, one shows directly that $\V$ satisfies that

\begin{equation}
H^0(\g_-,\mathbb V)=\mathbb E \quad\textrm{ and }\quad H^1(\g_-,\mathbb V)=\bigoplus_{j\in J}\circledcirc^{r_j}\g_{-1,j}^*\circledcirc\mathbb E.
\end{equation}

By relabelling we assume that $ H^1(\g_-,\mathbb V)=\bigoplus_{j=1}^d\circledcirc^{r_j}\g_{-1,j}^*\circledcirc\mathbb E$ with $r_1\leq...\leq r_d$.
Again, decomposing $\V$ with respect to the action of the grading element in $\z(\g_0)$ one obtains that
$$\V=\V_0\oplus...\oplus\V_N \quad\textrm{ such that } \quad\V_0=\E\quad\textrm{ and }\quad\g_j\V_i\subset \V_{i+j}.$$
Using Theorem \ref{Kostant} one deduces that $\ker(\square_{r_j})\cong \circledcirc^{r_i}\g_{-1,j}^*\circledcirc\mathbb E$ sits inside $\g_{-1,j}^*\otimes\V_{r_j-1}$ and  analogously as in Proposition \ref{U_iV_0} one therefore deduces that the $G_0$-equivariant maps

\begin{equation}\label{phi}
\phi_i: \V_i\rightarrow \U_{-i}(\g_-)^*\otimes \V_0 \quad \textrm{ given by }\quad v\mapsto(u\mapsto-u^{\top}v)
\end{equation} 
are inclusions, which are even isomorphisms provided that $i<r_1$.

\begin{rem}\label{Kr_j}
Denoting by $\K^{r_j}$ the kernel of the $G_0^{ss}$-equivariant projection 
\begin{equation}\label{general2}
\U_{-r_j}(\g_-)^*\otimes\E\rightarrow S^{r_j}\g_{-1}^*\otimes \E\rightarrow S^{r_i}\g_{-1,j}^*\otimes \E\rightarrow \circledcirc^{r_j}\g_{-1,j}^*\circledcirc\mathbb E.
\end{equation} one can prove as in Section \ref{onedim} that there are $G_0^{ss}$-equivariant isomorphisms
$$\phi_{i}: \V_i\cong \U_{-i}(\g_-)^*\otimes \E \quad\textrm{ for } 0\leq i<r_1$$
$$\phi_{i}: \V_i\cong \U_{-i}(\g_-)^*\otimes \E\cap(\U_{-(i-r_1)}(\g_-)^*\otimes\K^{r_1}) \quad\textrm{ for } r_1\leq i<r_2$$
$$:$$
$$\phi_i: \V_i\cong \U_{-i}(\g_-)^*\otimes \E\cap(\U_{-(i-r_1)}(\g_-)^*\otimes\K^{r_1})\cap...\cap (\U_{-(i-r_d)}(\g_-)^*\otimes\K^{r_d}) \textrm{ for } r_{d}\leq i\leq N.$$
\end{rem}
Choosing a principal connection $\nabla$ on $\G_0^{ss}$ and a splitting of the filtration of the tangent bundle $TM\cong\gr(TM)$, one may define a linear connection $\tilde\nabla$ of homogeneity $\geq 0$ on $V$ by $\tilde\nabla:=\nabla+\partial$ and constructs in the same way as in Section \ref{onedim} an operator $L:\Gamma(V_0)\rightarrow\Gamma(V)$ of weighted order $N$, which is characterised by the same properties as the analogous operator in Proposition \ref{Splittingop}. It is even given by the same formula (\ref{Lformula}). The same reasoning as in the proof of Proposition \ref{Splittingop} then shows that the composition $\sigma(L^i)\circ \phi_i$ equals $-id$ and hence $L^i$ induces a surjective vector bundle map

\begin{equation}\label{Lagain}
L^i: \J^i(V_0)\rightarrow V_0\oplus...\oplus V_i, 
\end{equation}
which is an isomorphism provided that $i<r_1$.
\\Denote by $\gr_{-1,j}(TM)$ the natural vector bundle corresponding to the irreducible representation $\g_{-1,j}$ and consider the linear differential operator given by
$$D^{\nabla}=(D_1^{\nabla},...,D_d^{\nabla}): \Gamma(E)\rightarrow\Gamma(\circledcirc^{r_1} \gr_{-1,1}(TM)^*\circledcirc E\oplus...\oplus\circledcirc^{r_d} \gr_{-1,d}(TM)^*\circledcirc E),$$
where $$D_j^{\nabla}:=-id\otimes\phi_{r_j-1}\circ\pi_j\circ\tilde\nabla\circ L: \Gamma(E)\rightarrow \Gamma(\circledcirc^{r_j} \gr_{-1,j}(TM)^*\circledcirc E)$$ and
$\pi_j: \gr(TM)^*\otimes V\rightarrow\gr_{-1,j}(TM)^*\otimes V_{r_j-1}\rightarrow\ker(\square_{r_j})$ is the natural projection.
The operator $D_j^{\nabla}$ is of weighted order $r_j$ with weighted symbol given by the natural projection $\U_{-r_j}(\gr(TM))^*\otimes E\rightarrow\circledcirc^{r_j} \gr_{-1,j}(TM)^*\circledcirc E$ induced by (\ref{general2}). We set $F_{j}:=\circledcirc^{r_j} \gr_{-1,j}(TM)^*\circledcirc E$ and $F:=F_{1}\oplus...\oplus F_{d}$.
\\If $D=(D_1,...,D_d): \Gamma(E)\rightarrow\Gamma(F_{1}\oplus...\oplus F_{d})$ is a differential operator, which differs from $D^{\nabla}$ by a bundle homomorphism
$\J^{r_1-1}(E)\rightarrow F$, then the isomorphism $L^{r_1-1}:\J^{r_1-1}(E)\cong V_0\oplus...\oplus V_{r_1-1}$ can be used to rewrite the system $D(s)=0$ step by step as in Section \ref{onedim} into a system of the form $\tilde\nabla(\Sigma)+C(\Sigma)=0$ for a bundle map $C: V\rightarrow T^*M\otimes V$.
%\\This shows in particular that we can deal with the first operators in BGG-sequences, which covers many geometric interesting equations.

\begin{rem}
Suppose that $D=(D_1,...,D_d): \Gamma(E)\rightarrow\Gamma(F_{1}\oplus...\oplus F_{d})$ is a linear differential operator such that $D_j$ is of the same weighted order and has the same weighted symbol as $D^{\nabla}_j$, then the approach of this article will  in general not apply straightforward, since $L^{i}$ is an isomorphism only for $i<r_1$. However, one may exploit things a bit further. Let us explain this in the case $d=2$:
\\For a linear differential operator $\phi:\J^r(W)\rightarrow\bar W$ of weighted order $r$ between vector bundles $W$ and $\bar W$ over a filtered manifold $M$, there is a general notion of its $\ell$-th prolongation $p_{\ell}(\phi): \J^{r+\ell}(W)\rightarrow \J^{\ell}(\bar W)$, which is given by $p_{r+\ell}(\phi)(j^{r+\ell}_xs)=j^{\ell}_x\phi(j^rs)$, see \cite{Morimoto3} or also \cite{N}. Correspondingly, the $\ell$-th prolongation $Q^{r+\ell}$ of the equation $Q^r:=\ker(\phi)$ associated to the operator $\phi$ is given by the kernel of $p_{\ell}(\phi)$. 
Note that a section $s$ of $W$ is a solution of $Q^r$, i.e. $\phi(j^rs)=0$, if and only if $s$ is a solution of $Q^{r+\ell}$ for all $\ell\geq 0$. In general $Q^{r+\ell}$ need not to be a vector bundle, but generically this will be the case.
\\Now by assumption on $D=(D_1,D_2)$, we may write 
$$D_1(s)=D_1^{\nabla}(s)+\psi_1(j^{r_1-1}s)\quad \textrm{ and } \quad D_2(s)=D_2^{\nabla}(s)+\psi_2(j^{r_2-1}s),$$ 
for some vector bundle maps $\psi_1: \J^{r_1-1}(E)\rightarrow F_1$ and  $\psi_2: \J^{r_2-1}(E)\rightarrow F_2$.
\\If $r:=r_1=r_2$, we know that $L^{r-1}$ defines an isomorphism $\J^{r-1}(E)\cong V_0\oplus...\oplus V_{r-1}$ and we can find vector bundle maps $A_i: V\rightarrow F_i$ such that $D_i(s)=D^{\nabla}_i(s)+A_i(Ls)$ for $i=1,2$. Hence the prolongation procedure of Section \ref{onedim} can be applied without problems to rewrite $D(s)=0$ into a system of the form $\tilde\nabla(\Sigma)+C(\Sigma)=0$.  
\\If $r_1<r_2$, the operator $L^{r_1-1}$ still defines an isomorphism $\J^{r_1-1}(E)\cong V_0\oplus...\oplus V_{r_1-1}$ and we can at least find a vector bundle map $A_1: V\rightarrow F_1$ such that $D_1(s)=D^{\nabla}_1(s)+A_1(Ls)$. Now let $Q^{r_1+\ell}\subset \J^{r_1+\ell}(E)$ be the $\ell$-th prolongation of $D_1(s)=0$ 
and denote by $P^{i}$ the kernel of the vector bundle map $L^{i}: \J^{i}(E)\rightarrow V_0\oplus....\oplus V_{i}$. 
We claim that $Q_x^{r_{2}-1}\cap P_x^{r_2-1}=\{0\}$ for all $x\in M$:
\\In fact, suppose that $j^{r_2-1}_xs\in Q_x^{r_{2}-1}\cap P_x^{r_2-1}$. The diagram (\ref{DiaSplit}) still holds and hence it follows that $j^i_xs\in P_x^{i}$ for all $i\leq r_2-1$. In particular, we have $j^{r_1}_xs\in P_x^{r_1}$ and, since $P_x^{r_1-1}=\{0\}$ by (\ref{Lagain}), this implies that $j^{r_1-1}_xs=0$. Hence $j^{r_1}_xs$ is an element in 
$\U_{-r_1}(\gr(T_xM))^*\otimes E_x\subset \J^{r_1}_x(E)$, which lies in the kernel of the weighted symbol of $L^{r_1}$. On the other hand, since $j^i_xs\in Q^{i}_x$ for all $i$, we obtain that
$j^{r_1}_xs$ also lies in the kernel of the weighted symbol of $D^1$. Since $\sigma(L^{r_1})\circ\phi_{r_1}=-id$, we deduce from remark \ref{Kr_j} that that the intersection of the kernels of these weighted symbols is zero and so $j^{r_1}_xs=0$. Hence $j^{r_1+1}_xs$ is an element of $\U_{-r_1-1}(\gr(T_xM))^*\otimes E_x$. In addition, it has to be in the kernel of the weighted symbol of $L^{r_1+1}$ and in the the kernel of the weighted symbol of the the first prolongation $p_1(D_1)$. In \cite{N} it was shown that the kernel of $\sigma_x(p_1(D_1))$ coincides with the intersection
$$\U_{-r_1-1}(\gr(T_xM))^*\otimes E_x\cap(\U_{-1}(\gr(T_xM))^*\otimes K^{r_1}_x).$$ Hence from $\sigma(L^{r_1+1})\circ\phi_{r_1+1}=-id$ and remark \ref{Kr_j} one concludes again that the intersection of the two kernels is zero and so $j^{r_1+1}_xs=0$. 
Since the kernel of $\sigma_x(p_\ell(D_1))$ equals $ \U_{-r_1-\ell}(\gr(T_xM))^*\otimes E_x\cap(\U_{-\ell}(\gr(T_xM))^*\otimes K^{r_1}_x)$, see \cite{N}, one shows in this way step by step that $j^{r_2-1}_xs=0$, which proves the claim.
\\Therefore the operator $L^{r_2-1}$ induces a fiberwise injective map 
\begin{equation}\label{Lisorem}
L^{r_2-1}|_{Q^{r_2-1}}: Q^{r_2-1}\rightarrow V_0\oplus...\oplus V_{r_2-1}, 
\end{equation}
which is a injective vector bundle map, if $Q^{r_2-1}$ is a vector bundle.
\\Choosing a splitting of this injection, shows that we can find a map $A_2: V\rightarrow F_2$ such that 
$$D(s)=0 \quad \textrm{ if and only if } \quad D_1^\nabla(s)+A_1(Ls)=0 \textrm{ and } D_2^\nabla(s)+A_2(Ls)=0,$$
since $D_1^\nabla(s)+A_1(Ls)=0$ implies that $j^{r_2-1}_xs\in Q^{r_2-1}_x$ for all $x\in M$.
The map $A=A_1+A_2$ is of homogeneity $\geq 1$ and so the prolongation procedure of Section \ref{onedim} can be applied to rewrite $D(s)=0$ into a system of the form $\tilde\nabla(\Sigma)+C(\Sigma)=0$.  
\\Finally, let us remark that if $Q^{i}$ is a vector bundle for all $i\leq r_2-1$ and all maps $Q^{i}\rightarrow Q^{i-1}$, which are induced from the projections $\pi^{i}_{i-1}: \J^i(E)\rightarrow \J^{i-1}(E)$, are surjective, then (\ref{Lisorem}) is even an isomorphism.
\end{rem}


\begin{thebibliography}{XX}

\bibitem{BG} Beals, R., Greiner, P.C.: \textit{Calculus on Heisenberg manifolds}. Annals of Math. Studies, vol. 119. Princeton Univ. Press, Princeton NJ. 1988.
\bibitem{Biquard} Biquard, O.: \textit{M\'etriques d'Einstein asymptotiquement sym\'etriques}. Ast\'erisque 265, 2000.
\bibitem{BCEG} Branson, T., \v Cap, A., Eastwood, M., Gover, A.R.: \textit{Prolongations of Geometric Overdetermined Systems}. Intern. Jour. of Math. vol. 17. no. 6. 2006. p. 641-664.
\bibitem{Bryant} Bryant, R.: \textit{Conformal geometry and 3-plane fields on 6-manifolds}. In: RIMS Symposium Proceedings: Developments of Cartan Geometry
and Related Mathematical Problems, vol. 1502 . July 2006. p. 1-15.
\bibitem{CD} Calderbank, D.M.J., Diemer, T.: \textit{Differential invariants and curved Bernstein-Gelfand-Gelfand sequences.} J. Reine Angew. Math. 537. 2001. p. 67-103.
\bibitem{Cap2} \v Cap, A.: \textit{Overdetermined Systems, Conformal Differential Geometry and the BGG Complex}. In: M.G. Eastwood, W. Miller (eds.): \textit{Symmetries and Overdetermined Systems of Partial Differential Equations}. The IMA Volumes in Mathematics and its Applications 144, Springer 2008.
\bibitem{CG} \v Cap, A., Gover, A.R.: \textit{Tractor calculi for parabolic geometries.} Trans. Amer. Math. Soc.
\bibitem{CSS} \v Cap, A., Slov\'ak, J., Sou\v cek, V.: \textit{Bernstein-Gelfand-Gelfand sequences.} Annals of Math. 154. 2001. p. 97-113.
\bibitem{CSbook} \v Cap, A., Slov\'ak, J.: \textit{Parabolic Geometries I: Background and General Theory.} Mathematical Surveys and Monographs Vol. 154, American Mathematical Society.
\bibitem{Cartan} Cartan, \'E.: \textit{Les syst\` emes de Pfaff  \`a cinq variables et les \'equations aux deriv\'ees partielles du second order}.
Ann. \'Ecole Normale 27. 1910. p. 263-355.
\bibitem{Dixmier} Dixmier, J.: \textit{Enveloping algebras}. Graduate Studies in Mathematics, Vol. 11. American Mathematical Society.

\bibitem{E} Eastwood, M.: \textit{Prolongation of Linear Overdetermined Systems on Affine and Riemannian Manifolds}. Rend. Circ. Mat. Palermo Suppl. 75. 2005. p. 89-108.
\bibitem{EG} Eastwood, M., Gover, A.R.: \textit{Prolongation on Contact Manifolds}. Available as preprint arXiv: 0910.5519.

\bibitem{HSSS} Hammerl, M., Somberg, P., Soucek, V., Silhan, J.: \textit{On a new normalization on tractor covariant derivatives.} Available as preprint arxiv 1003.6090.
\bibitem{Kost} Kostant, B.: \textit{Lie algebra cohomology and the generalized Borel-Weil theorem.} Annals of Math. 74. 1961. p. 329-387.
49 no. 2 .1977.p. 496-511.
\bibitem{Morimoto1} Morimoto, T.: \textit{Th\'eor\`eme de Cartan-K\"ahler dans une classe de fonctions formelles Gevrey.} C. R. Acad. Sci. Paris. 311, s\'erie A. 1990. p. 443-436.
\bibitem{Morimoto2} Morimoto, T.:\textit{ Th\'eor\`eme d'existence de solutions analytiques pour des syst\`emes d'\'equations aux d\'eriv\'ees partielles non-lin\'eaires avec singularit\'es}. C.R. Acad. Sci. Paris. 321, s\'erie 1. 1995. p.1491-1496.'
 \bibitem{Morimoto3} Morimoto, T.: \textit{Lie algebras, geometric
  structures and differential equations on filtered manifolds}. In ``Lie
  Groups Geometric Structures and Differential Equations - One Hundred
  Years after Sophus Lie'', Adv. Stud. Pure Math., Math.
  Soc. of Japan, Tokyo. 2002. 205-252.
\bibitem{Mor5} Morimoto, T.: \textit{Differential Equations Associated to a Representation of a Lie algebra from the Viewpoint of Nilpotent Analysis}. RIMS Kokyuroku 1502, Kyoto University.  2006/07.p. 238-250.
 \bibitem{N} Neusser, K.: \textit{Universal Prolongation of Linear Partial Differential Equations on Filtered Manifolds.} Arch. Math. (Brno) 45, issue 4. 2009. p. 289-300.
 \bibitem{thesis} Neusser, K.: \textit{Weighted Jet Bundles and Differential Operators for Parabolic Geometries}. PhD Thesis, University of Vienna, 2010. Avaliable online: http://othes.univie.ac.at/9768.
\bibitem{Spencer} Spencer, D.C.: \textit{Overdetermined systems of linear partial differential equations}. Bulletin of Amer. Math. Soc. 75. 1969. p. 179-239.
\bibitem{Tay}Taylor, M.E.: \textit{Noncommutative microlocal analysis I}. Mem. Amer. Math. Soc. 52, no. 313. 1984.
\bibitem{Yam} Yamaguchi, K: \textit{Differential Systems Associated with Simple Graded Lie Algebras}. Advanced Studies in Pure Math. 22 .1993. p. 413-494.
\end{thebibliography}
\end{document}